\documentclass[10pt]{article}

\usepackage{amssymb,latexsym,amsmath,enumerate,verbatim,amsfonts,amsthm}
\usepackage{color} % added by TK
\usepackage{cite}
\usepackage{rotating}
\usepackage{multirow}

\newtheorem{assumption}{Assumption}
\newtheorem{theorem}{Theorem}
\newtheorem{corollary}{Corollary}
\newtheorem{proposition}{Proposition}
\newtheorem{remark}{Remark}
\newtheorem{example}{Example}

\def\R{{\rm I\!R}}
\def\P{{\frak P}}
\def\D{{\frak D}}

\def\Argmin{\mathop{\rm Arg\,min}}

\def\argmin{\mathop{\rm arg\,min}}

\def\hat{\widehat}

\def\tilde{\widetilde}

\def\ra{\rangle}
\def\la{\langle}

\def\st{\stackrel}

\def\Argmin{\mathop{\rm Arg\,min}}
\def\argmin{\mathop{\rm arg\,min}}

\def\disp{\displaystyle}

\def\tto{\;{\lower 1pt \hbox{$\rightarrow$}}\kern -10pt
\hbox{\raise 2pt \hbox{$\rightarrow$}}\;}

\title{\sf Peaceman-Rachford splitting for a class of nonconvex optimization problems}

\author{
Guoyin Li \thanks{Department of Applied
Mathematics, University of New South Wales, Sydney 2052, Australia.
E-mail: {g.li@unsw.edu.au}. This author was partially supported by a research grant from Australian Research Council.}
\and
Tianxiang Liu\thanks{Department of Applied Mathematics, the Hong Kong Polytechnic University, Hong Kong. This author was supported partly by the AMSS-PolyU Joint Research Institute Postdoctoral Scheme. E-mail: {tiskyliu@polyu.edu.hk}.}
\and
Ting Kei Pong \thanks{Department of Applied Mathematics, the Hong Kong Polytechnic University, Hong Kong.
This author's work was supported partly by Hong Kong Research Grant Council PolyU253008/15p. E-mail: {tk.pong@polyu.edu.hk}.}
}

\date{January 9th, 2017}

\begin{document}
\maketitle
\begin{abstract}
We study the applicability of the Peaceman-Rachford (PR) splitting method for solving nonconvex optimization problems.
 When applied to minimizing the sum of a strongly convex Lipschitz differentiable function and a proper closed function, we show that if the
  strongly convex function has a large enough strong convexity modulus and the step-size parameter is chosen below a threshold that is computable,
   then any cluster point of the sequence generated, if exists, will give a stationary point of the optimization problem.
    We also give sufficient conditions guaranteeing boundedness of the sequence generated. We then discuss one way to split the objective so
     that the proposed method can be suitably applied to solving optimization problems with a coercive objective that is the sum of a (not necessarily strongly)
     convex Lipschitz differentiable function and a proper closed function; this setting covers a large
      class of nonconvex feasibility problems and constrained least squares problems. Finally, we illustrate the proposed algorithm numerically.
\end{abstract}

\section{Introduction}

Consider the following optimization problem with competing structure:
  \begin{equation}\label{P0}
    \min_u \ f(u) + g(u),
  \end{equation}
  where $f$ and $g$ are proper closed possibly nonconvex functions.  Optimization problems of this form arise in many important modern applications
  such as signal processing, machine learning and statistics \cite{BDE,Tib96,BvdBSC13,FanJ}.  A typical application of \eqref{P0} is to solve some ill-posed inverse problems
  where the function $f$ represents the data fitting term and the function $g$ is the regularization term.
  To solve problems with competing structures, an important and powerful class of algorithms is the class of splitting methods. In these methods, the objective function is decomposed into simpler individuals which are then processed separately in the subproblems. Two classical splitting methods in the literature are the Douglas-Rachford (DR) splitting method \cite{PDE,LionsM79,Eckstein_Bertsekas} and the Peaceman-Rachford (PR) splitting method \cite{PR55,LionsM79}.

 The PR splitting method was originally introduced in \cite{PR55} for solving linear heat flow equations, and was later generalized to deal with nonlinear equations in \cite{LionsM79}.
 In the case when $f$ and $g$ are both convex, the PR splitting method can be described conveniently by the following update:
  \begin{equation}\label{PRsplitting}
    x^{t+1} = (2{\rm prox}_{\gamma g} - I)\circ(2{\rm prox}_{\gamma f} - I)(x^t),
  \end{equation}
  where $I$ is the identity mapping, $\gamma > 0$ and
  \begin{equation*}%\label{proxh0}
  {\rm prox}_{\gamma h}(z) := \Argmin_u\left\{\gamma h(u) + \frac12 \|u - z\|^2\right\},
  \end{equation*}
  i.e., the set of minimizers of the problem $\min\limits_u  \gamma h(u) + \frac12 \|u - z\|^2$;
  we note that this set is a singleton when $h$ is convex.
Although the PR splitting method can be faster than the DR splitting method (see, for example, \cite{Boyd2015} and Example \ref{ex:N1} in Appendix), the PR splitting method was not as popular as the DR splitting method. This is also witnessed by the fact that the PR splitting method is not discussed nor mentioned in the recent monograph \cite{BauCom11} on operator splitting methods. One of the
main reasons for the unpopularity is that, even in the convex settings, the PR splitting method is not convergent in general.
To guarantee convergence, typically one would require either the operator $(2{\rm prox}_{\gamma f} - I)$ or $(2{\rm prox}_{\gamma g} - I)$ to be
a {\it contraction mapping}. In applications where $f$, $g$ are both convex, this requirement typically needs $f$ or $g$ to be strongly convex,
which largely limits the applicability of the PR splitting method; see, for example, \cite{Comb09,LionsM79}.
In contrast, under a commonly used constraint qualification which can be easily satisfied, the DR splitting method converges in the convex case \cite[Theorem 20]{Comb12}. Moreover, recently, it has been shown in \cite{LiPong14_2} that the DR splitting method can be adapted to a nonconvex setting with global convergence guaranteed under some assumptions.
This broadens the applicability of the DR splitting method to cover many nonconvex feasibility problems and many important nonconvex optimization problems arising in statistical machine learning such as the $\ell_{1/2}$ regularized least squares problem.

In this paper, to broaden the applicability of the PR splitting method, we extend it to a nonconvex setting.
By constructing a merit function which captures the progress of the PR splitting method, we extend the global convergence of the PR splitting method from the known convex setting to the case where the objective function can be decomposed as the sum of a
strongly convex Lipschitz differentiable function and a nonconvex function, under suitable assumptions.
%This then enables us to apply the PR splitting method to solving important nonconvex optimization problems such as the least squares optimization problems with nonconvex constraints and nonconvex feasibility problems.
As a by-product, this extension also allows us to establish the global convergence
and iteration complexity of a new PR splitting method for convex optimization problems in the {\em absence} of strong convexity. The underlying intuitive idea is that one can decompose a non-strongly convex function $F + G$ into the sum of a strongly convex function $f=F+\gamma \|\cdot\|^2$ and a nonconvex function $g=G-\gamma\|\cdot\|^2$, if a $\gamma > 0$ can be chosen so that $f$ is strongly convex.

  The contributions of this paper are two-fold. First, we establish that, for the sequence generated by the PR splitting method applied to minimizing the sum of a strongly convex Lipschitz differentiable function and a proper closed function, if the strongly convex function has a sufficiently large strong convexity modulus and the step-size parameter is chosen below a threshold that is computable, then any cluster point, if exists, gives a stationary point of the optimization problem. We also provide sufficient conditions to guarantee boundedness of the sequence generated. To our knowledge, this is the first work that studies the convergence of the PR splitting method for nonconvex optimization problems. Second, we demonstrate how the method can be suitably applied to minimizing a coercive function $F+G$, where $G$ is a proper closed function, and $F$ is convex Lipschitz differentiable but {\em not} necessarily strongly convex. Even in the case when $G$ is also convex, it was previously unknown in the literature how the PR splitting method can be suitably applied to solving it. Our study largely broadens the applicability of the PR splitting method. We also discuss global iteration complexity of this new PR splitting method under the additional assumption that $G$ is convex, and establish {\em global} linear convergence of the sequence generated if $F+G$ is further assumed to be strongly convex.

  The rest of the paper is organized as follows. In Section~\ref{sec2}, we fix the notation and recall some basic definitions which will be used throughout this paper. In Section~\ref{sec3}, we establish the convergence of the PR splitting method for nonconvex optimization problems where the objective function can be decomposed as the sum of a strongly convex function and a proper closed function, under suitable assumptions. In Section~\ref{sec4}, we demonstrate how the PR splitting method can be applied in the absence of strong convexity. In Section~\ref{sec5}, as applications, we illustrate how the PR splitting method can be applied to solving two important classes of nonconvex optimization problems that arise in the area of statistics and machine learning: constrained least squares
  problem and feasibility problems. We also demonstrate our approach numerically. Our concluding remarks are in Section~\ref{sec6}. Finally, in the Appendix, we provide simple and concrete examples illustrating
  the different behaviors of the classical PR splitting method, the classical DR splitting method and our proposed PR splitting method.

\subsection{Notation}\label{sec2}

In this paper, the $n$-dimensional Euclidean space is denoted by $\R^n$, with the associated inner product denoted by $\langle\cdot,\cdot\rangle$
and the induced norm denoted by $\|\cdot\|$.
For an extended-real-valued function $f:\R^n \to (-\infty,\infty]$, we say that $f$ is proper if it is never $-\infty$ and its domain, ${\rm dom}f:=\{x \in \R^n: f(x)<+\infty\}$, is nonempty.
Such a function is said to be closed if it is lower semicontinuous.
For a proper function $f$, we let $z\st{f}{\to}x$ denote $f(z)\to f(x)$ and $z\to x$.
The limiting {\em subdifferential} of $f$ at $x\in\mathrm{dom}\,f$ is defined by \cite{Rock98}
\begin{equation}\label{ls}
\partial f(x):=\left\{v\in\R^n :\;\exists x^t\st{f}{\to}x,\;v^t\to v\;\mbox{ with }\disp\liminf_{z\to x^t}\frac{f(z)-f(x^t)-\la v^t,z-x^t\ra}{\|z-x^t\|}\ge 0\mbox{ for each }t\right\}.
\end{equation}
From the above definition, one immediately obtains the following robustness property:
\begin{equation}\label{outersemi}
  \left\{v\in \R^n:\; \exists x^t \st{f}{\to}x,\; v^t\to v\;, v^t\in \partial f(x^t)\right\} \subseteq \partial f(x).
\end{equation}
%We also use the notation ${\rm dom}\,\partial f := \{x\in \R^n:\; \partial f(x)\neq \emptyset\}$.
The subdifferential \eqref{ls} reduces to the derivative of $f$ (denoted by $\nabla f$) if $f$ is continuously differentiable, and the classical subdifferential in convex analysis if $f$ is convex (see, for example, \cite[Proposition~8.12]{Rock98}).
For a function $f$ having more than one group of variables, we let $\partial_x f$ (resp., $\nabla_x f$) denote the subdifferential (resp., derivative) of $f$ with respect to the variable $x$.

We say that a function $f$ is a {\em strongly convex function} with modulus $\sigma>0$ if $f-\frac{\sigma}{2}\|\cdot\|^2$ is a convex function. A function $f$ is said to be coercive if $\liminf\limits_{\|x\|\to\infty}f(x) = \infty$.
%We state below an auxiliary lemma concerning strongly convex functions, which is a folklore. We include a simple proof for the convenience of the readers.
%
%\begin{lemma}\label{strongh}
%  Let $h$ be a differentiable strongly convex function. Then $\liminf\limits_{\|x\|\to \infty}\|\nabla h(x)\| = \infty$.
%  If, in addition, $\nabla h$ is Lipschitz continuous with modulus at most $L>0$, then
%  \[
%  \inf_x \left\{ h(x) - \frac1{2L} \|\nabla h(x)\|^2 \right\} > -\infty.
%  \]
%\end{lemma}
%\begin{proof}
%  As $h$ is strongly convex, a (unique) minimizer for $\min\limits_{x \in \mathbb{R}^n} h(x)$  exists (say $x^*$).  Thus, there exists $\sigma>0$ such that for any $x \in \mathbb{R}^n$,
%  \[
%  h(x)\ge h(x^*) + \frac\sigma 2 \|x - x^*\|^2.
%  \]
%  Combining this with the definition of convex conjugate gives ${\rm dom}(h^*) = \R^n$, i.e., $h$ is co-finite. The first conclusion now follows from \cite[Lemma~26.7]{Roc70}.
%
%  Now suppose in addition that $\nabla h$ is Lipschitz continuous with modulus at most $L>0$. Then we have from the optimality of $x^*$ and the Taylor's theorem that for any $x$
%  \[
%  \begin{split}
%  h(x^*) &\le h\left(x - \frac1L \nabla h(x)\right) \le h(x) - \left\langle\nabla h(x),\left(x - \frac1L \nabla h(x)\right) - x \right\rangle + \frac{L}{2}\left\|\left(x - \frac1L \nabla h(x)\right) - x\right\|^2\\
%  & = h(x) - \frac1{2L}\|\nabla h(x)\|^2.
%  \end{split}
%  \]
%  The second conclusion follows immediately. This completes the proof.
%\end{proof}
For a nonempty closed set $S \subseteq \R^n$, its indicator function $\delta_S$ is defined by
\[
\delta_S(x)=\begin{cases}
0 & \mbox{ if }  x \in S, \\
+\infty & \mbox{ if }  x \notin S.
\end{cases}
\]
%The limiting normal cone of $S$ at $x \in S$ is given by $N_S(x):=\partial \delta_S(x)$.
We use the notation $d_S(x)$ or ${\rm dist}(x,S)$ to denote the distance from an $x\in \R^n$ to $S$, i.e., $d_S(x):= \inf_{y\in S}\|x - y\|$. Moreover, we use
$P_S(x)$ to denote the points in $S$ that are closest to $x$: note that $P_S(x)$ is a singleton set if $S$ is, in addition, convex.

Finally, for an optimization problem $\displaystyle \min_{x \in \R^n} f(x)$, we use $\displaystyle \Argmin_{x} f(x)$ to denote the set consisting of all
its minimizers. If $\displaystyle \Argmin_{x} f(x)$ turns out to be a singleton, we simply denote it as $\displaystyle \argmin_{x} f(x)$.

%{\color{red} TK: I am not defining KL. We don't really need that.}

\section{Peaceman-Rachford splitting for structured nonconvex problems}\label{sec3}

Recall that the class of problems we consider is
  \begin{equation}\label{P1}
    \min_u \ f(u) + g(u),
  \end{equation}
  where $f$ and $g$ are proper closed possibly nonconvex functions. As discussed in the introduction, even in the case when both $f$ and $g$ are convex, typically one would need $f$ (or $g$) to be strongly convex to guarantee convergence of the PR splitting method. Moreover, we recall that the Lipschitz differentiability of $f$ played an important role in the recent convergence analysis of the closely related DR splitting method in \cite{LiPong14_2} for \eqref{P1} in the nonconvex settings.
  Motivated by these, we make the following blanket assumption on $f$ throughout this paper.
  \begin{assumption}[{{\bf Blanket assumption on $f$}}]
    The function $f$ is strongly convex with a strong convexity modulus at least $\sigma > 0$, and is Lipschitz differentiable so that $\nabla f$ has a Lipschitz continuity modulus at most $L > 0$.
  \end{assumption}
  Notice that the proximal mapping ${\rm prox}_{\gamma f}(z)$ of a strongly convex function $f$ is well defined for any $\gamma > 0$ at any point $z$. Thus, in order for the iterates in \eqref{PRsplitting} to be well defined, we only need to make additionally the following blanket assumption on $g$ in this paper.
  \begin{assumption}[{{\bf Blanket assumption on $g$}}]\label{assumg}
    The function $g$ is proper closed with a nonempty proximal mapping ${\rm prox}_{\gamma g}(z)$ for any $z$ and for the $\gamma > 0$ we use in the algorithm.
  \end{assumption}

Under the blanket assumptions, we consider the following adaptation of the PR splitting method to solve the possibly nonconvex problem \eqref{P1}, which can be easily shown to be equivalent to \eqref{PRsplitting} in the case when $f$ and $g$ are convex (so that the proximal mappings are single-valued).\\ \

\fbox{\parbox{5.7 in}{
\begin{description}
\item {\bf PR splitting method}

\item[Step 0.] Input $x^0$ and $\gamma > 0$.

\item[Step 1.] Set
\begin{equation}\label{scheme}
\left\{
\begin{split}
&y^{t+1}=\argmin_{y} \left\{f(y) + \frac{1}{2\gamma}\|y - x^t\|^2\right\}, \\
&z^{t+1}\in\Argmin_{z} \left\{g(z) + \frac{1}{2\gamma}\|2y^{t+1} - x^t - z\|^2\right\},\\
&x^{t+1}=x^t+2(z^{t+1}- y^{t+1}).
\end{split}
\right.
\end{equation}

\item[Step 2.] If a termination criterion is not met, go to Step 1.
\end{description}
}}
\vspace{.1 in}

Our convergence analysis follows a similar line of arguments (with some intricate modifications) for showing convergence for the Douglas-Rachford splitting method as in our recent work \cite{LiPong14_2}, and has to make extensive use of the following merit function:
\begin{eqnarray}
    \P_\gamma(y,z,x) &:=& f(y) + g(z) - \frac3{2\gamma}\|y - z\|^2 + \frac1\gamma \langle x-y,z-y\rangle \label{Ddef} 	\\
    & =& \D_\gamma(y,z,x) - \frac1\gamma\|y - z\|^2, \nonumber
\end{eqnarray}
where $\D_\gamma$ is the so-called Douglas-Rachford merit function given by $\D_\gamma(y,z,x)=f(y) + g(z) - \frac1{2\gamma}\|y - z\|^2 + \frac1\gamma \langle x-y,z-y\rangle$ (see \cite[Definition~2.1]{LiPong14_2}), motivated by \cite[Eq.~35]{PaStBe14}.
%After the convergence analysis, we will discuss how such an algorithm can be suitably applied to solving \eqref{P1} when the natural choice of $f$ is not strongly convex.

Before proceeding, we make two important observations. First, %using \cite[Equation~13]{LiPong14_2} and \eqref{Ddef},
it is not hard to see that the merit function $\P_\gamma$ can alternatively be written as
  \begin{equation}\label{Drelation}
  \begin{split}
    \P_\gamma(y,z,x) & = f(y) + g(z) + \frac1{2\gamma}\|2y - z - x\|^2 - \frac{1}{2\gamma}\|x-y\|^2 - \frac2{\gamma}\|y - z\|^2\\
    & = f(y) + g(z) + \frac1{2\gamma}(\|x - y\|^2 - \|x - z\|^2-2\|y - z\|^2),
  \end{split}
  \end{equation}
  where the first relation follows from the elementary relation $\langle u,v\rangle = \frac12 (\|u + v\|^2 - \|u\|^2 - \|v\|^2)$ applied with $u = x-y$ and $v = z -y$ in \eqref{Ddef}, while the second relation is obtained by using the elementary relation $\langle u,v\rangle = \frac12 (\|u\|^2 + \|v\|^2 - \|u - v\|^2)$ in \eqref{Ddef} with $u = x-y$ and $v = z -y$.
We will make use of these equivalent formulations in the convergence analysis.
Second, we note by using the optimality conditions for the $y$ and $z$-updates in \eqref{scheme} that:
\begin{equation}\label{fyupdate}
  \begin{split}
    0 & = \nabla f(y^{t+1}) + \frac{1}{\gamma}(y^{t+1} - x^t),\\
    0 & \in \partial g(z^{t+1}) + \frac{1}{\gamma}(z^{t+1} - y^{t+1}) - \frac{1}{\gamma}(y^{t+1} - x^t),
  \end{split}
\end{equation}
where we made use of the subdifferential calculus rule \cite[Exercise~8.8]{Rock98}.
Consequently, for all $t\ge 1$,
\begin{equation}\label{optcond}
  0 \in  \nabla f(y^t) + \partial g(z^t) + \frac{1}{\gamma}(z^t - y^t).
\end{equation}
To establish convergence and characterize the cluster point of the sequence generated, we will subsequently show that $\lim_{t\rightarrow \infty}\|z^{t} - y^{t}\| = 0$ and that $g$ is ``continuous" at the cluster point along the sequence generated.

We are now ready to state and prove a convergence result for the PR splitting method \eqref{scheme}. We would like to point out that our proof is following exactly the same line of arguments as \cite[Theorem~1]{LiPong14_2}. However, there are two crucial differences. First, we now make use of the merit function \eqref{Ddef} in place of the Douglas-Rachford merit function. Second, as we will see in the upper estimate in \eqref{useful2}, the factor of $\gamma$ in the denominator is canceled, and thus the strong convexity modulus $\sigma$ comes into play in establishing the non-increasing property of the sequence $\{\P_\gamma(y^t,z^t,x^t)\}_{t\ge 1}$.

\begin{theorem}[{{\bf Global subsequential convergence}}]\label{thm1}
  Suppose that $3\sigma > 2L$ and the parameter $\gamma$ is chosen so that
  \begin{equation}\label{gammacond}
  0 <  \gamma < \frac{3\sigma  -2L}{L^2}.
  \end{equation}
  Then the sequence $\{\P_\gamma(y^t,z^t,x^t)\}_{t\ge 1}$ is nonincreasing. Moreover, if a cluster point $(y^*,z^*,x^*)$ of the sequence exists, then we have
  \begin{equation}\label{cond1}
    \lim_{t\rightarrow \infty}\|x^{t+1} - x^t\| = 2\lim_{t\rightarrow \infty}\|z^{t+1} - y^{t+1}\| = 0,
  \end{equation}
  the cluster point satisfies $z^* = y^*$, and
  \[
  0\in \nabla f(z^*) + \partial g(z^*).
  \]
\end{theorem}
\begin{remark}
  We note that the condition $3\sigma > 2L$ indicates that this convergence result can only be applied when $f$ has a relatively large strong convexity modulus, i.e., when $\sigma > \frac23 L$. It seems restrictive at first glance, but we will demonstrate in the next section how this theorem can be applied in a wide range of problems that do not explicitly contain a strongly convex part in the objective. Specifically, we will show that the method can be suitably applied to minimizing a coercive function $F+G$, where $G$ is a proper closed function and $F$ is convex Lipschitz differentiable but {\em not} necessarily strongly convex; see Corollary~\ref{cor:1}.
\end{remark}
\begin{proof}
  We study the behavior of $\P_\gamma$ along the sequence generated from the PR splitting method. First, using \eqref{Ddef} and the definition of the $x$-update,
  we see that
  \begin{equation}\label{D1ineq}
    \P_\gamma(y^{t+1},z^{t+1},x^{t+1}) - \P_\gamma(y^{t+1},z^{t+1},x^t) = \frac1\gamma\langle x^{t+1}-x^t,z^{t+1} - y^{t+1}\rangle = \frac1{2\gamma}\|x^{t+1} - x^t\|^2.
  \end{equation}
  Second, making use of the first relation in \eqref{Drelation} and the definition of $z^{t+1}$ as a minimizer, we have
  \begin{equation}\label{D2ineq}
    \begin{split}
      &\P_\gamma(y^{t+1},z^{t+1},x^t) - \P_\gamma(y^{t+1},z^t,x^t) \\
      & = g(z^{t+1}) + \frac1{2\gamma}\|2y^{t+1} - z^{t+1} - x^t\|^2 - \frac2\gamma\|y^{t+1} - z^{t+1}\|^2\\
      &\ \ \  - g(z^t) - \frac1{2\gamma}\|2y^{t+1} - z^t - x^t\|^2  + \frac2\gamma\|y^{t+1} - z^t\|^2\\
      & \le \frac2\gamma\left( \|y^{t+1} - z^t\|^2 - \|y^{t+1} - z^{t+1}\|^2\right)= \frac2\gamma\left( \|y^{t+1} - z^t\|^2 - \frac14\|x^{t+1} - x^t\|^2\right),
    \end{split}
  \end{equation}
  where the last relation is due to the definition of $x^{t+1}$. Consequently, summing \eqref{D1ineq} and \eqref{D2ineq}, we have
  \begin{equation}\label{D12ineq}
    \P_\gamma(y^{t+1},z^{t+1},x^{t+1}) - \P_\gamma(y^{t+1},z^t,x^t) \le \frac2\gamma \|y^{t+1} - z^t\|^2.
  \end{equation}
  Next, making use of the second relation in \eqref{Drelation}, we see that
  \begin{equation}\label{D3ineq}
    \begin{split}
      &\P_\gamma(y^{t+1},z^t,x^t) - \P_\gamma(y^t,z^t,x^t)\\
      & = f(y^{t+1}) + \frac1{2\gamma} \|x^t - y^{t+1}\|^2 - f(y^t) - \frac1{2\gamma} \|x^t - y^t\|^2 - \frac1{\gamma}\|y^{t+1} - z^t\|^2 + \frac1{\gamma}\|y^t - z^t\|^2\\
      & \le -\frac{1}{2}\left(\frac1\gamma + \sigma\right)\|y^{t+1} - y^t\|^2 - \frac1{\gamma}\|y^{t+1} - z^t\|^2 + \frac1{\gamma}\|y^t - z^t\|^2,
    \end{split}
  \end{equation}
  where, in the last inequality, we used the definition of $y^{t+1}$ as a minimizer and the strong convexity of the objective in the minimization problem that defines the $y$-update.
  Combining \eqref{D3ineq} with \eqref{D12ineq} gives further that
  \begin{equation}\label{D3ineq2}
    \P_\gamma(y^{t+1},z^{t+1},x^{t+1}) - \P_\gamma(y^t,z^t,x^t)\le  -\frac{1}{2}\left(\frac1\gamma +\sigma\right)\|y^{t+1} - y^t\|^2 + \frac1{\gamma}\|y^{t+1} - z^t\|^2 + \frac1{\gamma}\|y^t - z^t\|^2.
  \end{equation}

  To further upper estimate \eqref{D3ineq2}, observe from the first relation in \eqref{fyupdate} that
  \[
  \nabla f(y^{t+1}) =  \frac{1}{\gamma}(x^t - y^{t+1}).
  \]
  Since $f$ is strongly convex with modulus $\sigma > 0$ by assumption, we see that for all $t\ge 1$,
  \begin{equation*}
    \begin{split}
      &\left\langle \frac{1}{\gamma}(x^t - y^{t+1}) - \frac{1}{\gamma}(x^{t-1} - y^t),y^{t+1} - y^t\right\rangle \ge \sigma \|y^{t+1} - y^t\|^2\\
      &\Longrightarrow
      \langle x^{t} -x^{t-1},y^{t+1} - y^t\rangle \ge (1 +\gamma \sigma)\|y^{t+1} - y^t\|^2.
    \end{split}
  \end{equation*}
  Thus, making use of the definition of $x^t$ and the above relation, we obtain further that
  \begin{equation}\label{bdy}
    \begin{split}
      &\|y^{t+1} - z^t\|^2 = \|y^{t+1} - y^t + y^t - z^t\|^2 = \left\|y^{t+1} - y^t - \frac12(x^t - x^{t-1})\right\|^2\\
      & = \|y^{t+1} - y^t\|^2 - \langle y^{t+1} - y^t,x^t - x^{t-1}\rangle + \frac14\|x^t - x^{t-1}\|^2\\
      & \le -\gamma \sigma\|y^{t+1} - y^t\|^2 + \frac14\|x^t - x^{t-1}\|^2.
    \end{split}
  \end{equation}
  In addition, observe also from the definition of the $x$-update, the first relation in \eqref{fyupdate} and the Lipschitz continuity of $\nabla f$ that for $t\ge 1$
  \begin{equation}\label{ineqbd}
  2\|y^t - z^t\| = \|x^t - x^{t-1}\| \le (1 + \gamma L)\|y^{t+1} - y^t\|.
  \end{equation}
  Combining \eqref{bdy}, \eqref{ineqbd} with \eqref{D3ineq2}, we conclude that for any $t\ge 1$
  \begin{equation}\label{useful2}
  \begin{split}
    \P_\gamma(y^{t+1},z^{t+1},x^{t+1}) - \P_\gamma(y^t,z^t,x^t) &\le \frac{1}{2\gamma}\left((1 + \gamma L)^2 - 3\gamma \sigma - 1\right)\|y^{t+1} - y^t\|^2\\
    & = \frac{1}{2}\left(-3\sigma + 2L + \gamma L^2\right)\|y^{t+1} - y^t\|^2.
  \end{split}
  \end{equation}
  By our choice of $\gamma$, $-3\sigma + 2L + \gamma L^2 < 0$. From this we see immediately that $\{\P_\gamma(y^t,z^t,x^t)\}$ is nonincreasing.
  Summing \eqref{useful2} from $t=1$ to $N-1\ge 1$, we obtain that
  \begin{equation}\label{keybd}
    \P_\gamma(y^N,z^N,x^N) - \P_\gamma(y^1,z^1,x^1) \le \frac{1}{2}\left(-3\sigma + 2L + \gamma L^2\right)\sum_{t=1}^{N-1}\|y^{t+1} - y^t\|^2.
  \end{equation}
  Using this, the closedness of $\P_\gamma$ and the existence of cluster points,
  we conclude immediately from \eqref{keybd} that $\lim\limits_{t\rightarrow \infty}\|y^{t+1} - y^t\|=0$. Combining this with \eqref{ineqbd}, we conclude that \eqref{cond1} holds.
  Furthermore, combining these with the third relation in \eqref{scheme}, we obtain further that $\lim\limits_{t\rightarrow \infty}\|z^{t+1} - z^t\|=0$.

  Consequently, if $(y^*,z^*,x^*)$ is a cluster point of $\{(y^t,z^t,x^t)\}$ with a convergent subsequence $\{(y^{t_j},z^{t_j},x^{t_j})\}$ such that
  $\lim\limits_{j\rightarrow \infty}(y^{t_j},z^{t_j},x^{t_j}) = (y^*,z^*,x^*)$, then we must have
  \begin{equation}\label{limit}
  \lim_{j\rightarrow \infty}(y^{t_j},z^{t_j},x^{t_j}) =
  \lim_{j\rightarrow \infty}(y^{t_j-1},z^{t_j-1},x^{t_j-1}) = (y^*,z^*,x^*).
  \end{equation}
  Since $z^t$ is a minimizer of the subproblem,
  \begin{equation*}%\label{glimit}
  g(z^t) + \frac1{2\gamma}\|2y^t - z^t - x^{t-1}\|^2 \le g(z^*) + \frac1{2\gamma}\|2y^t - z^* - x^{t-1}\|^2.
  \end{equation*}
  Taking limit along the convergent subsequence
  and using \eqref{limit} yields
  \begin{equation*}%\label{glimit2}
  \limsup_{j\rightarrow \infty}g(z^{t_j})\le g(z^*).
  \end{equation*}
  Conversely, we have $\liminf\limits_{j\rightarrow \infty}g(z^{t_j})\ge g(z^*)$ by the lower semicontinuity of $g$.
  Thus,
  \begin{equation}\label{cond2}
    \lim_{j\rightarrow \infty}g(z^{t_j})= g(z^*).
  \end{equation}
  Using \eqref{outersemi}, \eqref{cond1}, \eqref{cond2} and passing to the limit in \eqref{optcond} along the convergent subsequence above, we conclude that the cluster point gives a stationary point of \eqref{P1}, i.e., $y^* = z^*$ and
\[
0\in \nabla f(z^*) + \partial g(z^*).
\]
This completes the proof.
\end{proof}

In the next theorem, we study sufficient conditions to guarantee boundedness of the sequence generated from the PR splitting method. Thus, a cluster point will necessarily exist under these conditions. %The proof follows a similar line of arguments as the proof of \cite[Theorem~4]{LiPong14_2}.

\begin{theorem}[{{\bf Boundedness of sequence}}]\label{thm2}
  Suppose that $3\sigma > 2L$ and the $\gamma$ is chosen to satisfy \eqref{gammacond}. Suppose in addition that $f + g$ is coercive, i.e., $\liminf_{\|u\|\to \infty} (f + g)(u) = \infty$.
  Then the sequence $\{(y^t,z^t,x^t)\}$ generated from \eqref{scheme} is bounded.
\end{theorem}
\begin{proof}
  Recall from Theorem~\ref{thm1} that the merit function is nonincreasing along the sequence generated from \eqref{scheme}. In particular,
  \begin{equation}\label{bd}
    \P_\gamma(y^t,z^t,x^t) \le \P_\gamma(y^1,z^1,x^1)
  \end{equation}
  whenever $t\ge 1$, where
  \begin{equation}\label{ineq2}
   \P_\gamma(y^t,z^t,x^t) = f(y^t) + g(z^t) - \frac1{2\gamma}\|x^t - z^t\|^2 + \frac1{2\gamma}\|x^t - y^t\|^2 - \frac1\gamma\|y^t-z^t\|^2
  \end{equation}
  from the second relation in \eqref{Drelation}. Next, recall from the definition of $x$-update that $x^t = x^{t-1} + 2(z^t - y^t)$, which together with the first relation in \eqref{fyupdate} gives
  \begin{equation}\label{rel0}
    \nabla f(y^t) = \frac1\gamma(x^{t-1}-y^t) = \frac1\gamma([x^t - z^t] - [z^t - y^t]).
  \end{equation}
  Moreover, for the function $f$ whose gradient is Lipschitz continuous with modulus $L$, we have
  \begin{equation}\label{Lipineq}
  f(z^t) \le f(y^t) + \langle \nabla f(y^t), z^t - y^t\rangle + \frac{L}{2}\|z^t - y^t\|^2.
  \end{equation}
  Combining these with \eqref{ineq2} and \eqref{bd}, we see further that
  \begin{equation}\label{bd1}
    \begin{split}
      &\P_\gamma(y^1,z^1,x^1) \ge f(y^t) + g(z^t) - \frac1{2\gamma}\|x^t - z^t\|^2 + \frac1{2\gamma}\|x^t - y^t\|^2 - \frac1\gamma\|y^t-z^t\|^2\\
      & \ge f(z^t) + g(z^t) - \langle \nabla f(y^t), z^t - y^t\rangle - \frac1{2\gamma}\|x^t - z^t\|^2 + \frac1{2\gamma}\|x^t - y^t\|^2 - \left(\frac{L}{2} + \frac1\gamma\right)\|y^t-z^t\|^2\\
      & = f(z^t) + g(z^t) - \frac1\gamma\langle x^t - z^t, z^t - y^t\rangle - \frac1{2\gamma}\|x^t - z^t\|^2 + \frac1{2\gamma}\|x^t - y^t\|^2 - \frac{L}{2}\|y^t-z^t\|^2\\
      & = f(z^t) + g(z^t) + \frac12\left(\frac1\gamma - L\right)\|y^t - z^t\|^2,
    \end{split}
  \end{equation}
  where the second inequality follows from \eqref{Lipineq}, the first equality follows from \eqref{rel0}, while the last equality follows from the elementary relation $\langle u,v\rangle = \frac12(\|u + v\|^2 - \|u\|^2 - \|v\|^2)$ applied to $u = x^t - z^t$ and $v = z^t - y^t$. From \eqref{bd1}, the coerciveness of $f+g$ and the fact that $\gamma < \frac{3\sigma - 2L}{L^2} \le \frac1L$, we conclude that $\{z^t\}$ and $\{y^t\}$ are bounded. The boundedness of $\{x^t\}$ now follows from these and the first relation in \eqref{fyupdate}. This completes the proof.
\end{proof}

\begin{remark}[{{\bf Comments on the proof of Theorem~\ref{thm2}}}]

\begin{enumerate}[{\rm (i)}]
  \item The technique of using \eqref{Lipineq} for establishing \eqref{bd1} was also used previously in \cite[Lemma~3.3]{Hong14} for
   showing that the augmented Lagrangian function is bounded below along the sequence generated from the alternating direction method of multipliers for a special class of problems.
   Here, we applied the technique to the new merit function $\P_{\gamma}$.
  \item The same technique used here can be applied to establishing the boundedness of the sequence generated by the DR splitting method studied in \cite{LiPong14_2}
  under a condition which is slightly weaker than the one used in \cite{LiPong14_2}. In fact, one can show that, the DR splitting method in \cite{LiPong14_2}
  generates a bounded sequence under the blanket assumptions of $f$ and $g$ in \cite[Section~3]{LiPong14_2}, the condition that $f + g$ is coercive and the choice of parameter specified in \cite[Theorem~4]{LiPong14_2} \footnote{This slightly improves \cite[Theorem~4]{LiPong14_2}
  because \cite[Theorem~4]{LiPong14_2} assumed a slightly stronger condition that $f$ and $g$ are bounded below and one of them is coercive.}.

      To see this, recall that for the DR splitting method, we also have $\nabla f(y^t) = \frac1\gamma(x^{t-1} - y^t)$ but have $x^t = x^{t-1} + (z^{t} - y^t)$ instead of the third relation in \eqref{scheme}. Thus, $\nabla f(y^t) = \frac1\gamma(x^t - z^t)$ and we have the following estimate for the DR merit function, making use of \eqref{Lipineq}:
      \begin{equation*}
        \begin{split}
          &\D_\gamma(y^t,z^t,x^t)  = f(y^t) + g(z^t) - \frac1{2\gamma}\|x^t - z^t\|^2 + \frac1{2\gamma}\|x^t - y^t\|^2\\
          & \ge f(z^t) + g(z^t) - \langle \nabla f(y^t), z^t - y^t\rangle - \frac{L}{2}\|z^t - y^t\|^2 - \frac1{2\gamma}\|x^t - z^t\|^2 + \frac1{2\gamma}\|x^t - y^t\|^2\\
          & = f(z^t) + g(z^t) - \frac1\gamma\langle x^t - z^t, z^t - y^t\rangle - \frac{L}{2}\|z^t - y^t\|^2 - \frac1{2\gamma}\|x^t - z^t\|^2 + \frac1{2\gamma}\|x^t - y^t\|^2\\
          & = f(z^t) + g(z^t) + \frac12\left(\frac1\gamma - L\right)\|y^t - z^t\|^2,
        \end{split}
      \end{equation*}
      where the last equality follows from the elementary relation $\langle u,v\rangle = \frac12(\|u + v\|^2 - \|u\|^2 - \|v\|^2)$ applied to $u = x^t - z^t$ and $v = z^t - y^t$. The boundedness of the sequence can then be deduced under the choice of $\gamma$ in \cite[Theorem~4]{LiPong14_2}, which guarantees $\gamma < \frac1L$, and the assumption that $f + g$ is coercive.
\end{enumerate}
\end{remark}

As in \cite[Theorem~4]{LiPong14} and \cite[Theorem~2]{LiPong14_2}, one can also show that the whole sequence generated is convergent under the additional assumption that $\P_\gamma(y,z,x)$ is a KL function.\footnote{We refer the readers to, for example, \cite{AtBoReSo10,AtBoSv13,BolDanLew07,BolDanLewShi07}, for the definition and examples of KL functions. In particular, if $f$ and $g$ are proper closed semi-algebraic functions, then $\P_\gamma$ is a KL function for any $\gamma > 0$.} To this end, note that for any $t\ge 1$, we have from \eqref{Ddef} and the third relation in \eqref{scheme} that
\begin{equation}\label{nablaxP}
  \nabla_x\P_\gamma(y^t,z^t,x^t) = \frac1\gamma (z^t - y^t) = \frac1{2\gamma}(x^t - x^{t-1}).
\end{equation}
Moreover, using the second relation in \eqref{Drelation}, one can obtain
\begin{equation}\label{nablayP}
\nabla_y\P_\gamma(y^t,z^t,x^t) = \nabla f(y^t) + \frac1\gamma(y^t - x^t) - \frac2\gamma(y^t - z^t)
  = \frac1\gamma(x^{t-1} - x^t) - \frac2\gamma(y^t - z^t) = 0
\end{equation}
where the second equality follows from the first relation in \eqref{fyupdate}, and the last equality follows again from  the third relation in \eqref{scheme}. Finally,
using the second relation in \eqref{Drelation}, one can compute that
\begin{equation}\label{nablazP}
  \begin{split}
    &\partial_z \P_\gamma(y^t,z^t,x^t) = \partial g(z^t) - \frac1\gamma(z^t - x^t) - \frac2\gamma(z^t - y^t)\\
    & = \partial g(z^t) + \frac1\gamma(z^t - y^t) - \frac1\gamma(y^t - x^{t-1}) - \frac1\gamma(z^t - y^t) + \frac1\gamma(y^t - x^{t-1}) - \frac1\gamma(z^t - x^t) - \frac2\gamma(z^t - y^t)\\
    &\ni -\frac4\gamma(z^t - y^t) + \frac1\gamma(x^t - x^{t-1}) = -\frac1\gamma(x^t - x^{t-1}),
  \end{split}
\end{equation}
where the inclusion follows from the second relation in \eqref{fyupdate} and the last equality follows from the third relation in \eqref{scheme}. Consequently, by combining \eqref{nablaxP}, \eqref{nablayP}, \eqref{nablazP} and \eqref{ineqbd}, we see the existence of $\kappa > 0$ so that
\begin{equation*}
  {\rm dist}\,(0,\partial \P_\gamma(y^t,z^t,x^t))\le \kappa \|y^{t+1} - y^t\|.
\end{equation*}
Using this, \eqref{useful2} and following the arguments as in the proof of \cite[Theorem~2]{LiPong14_2}, it is not hard to prove the following result. We omit the detailed proof here.

\begin{theorem}[{{\bf Global convergence of the whole sequence}}]\label{thm3}
  Suppose that $3\sigma > 2L$, the parameter $\gamma > 0$ is chosen as in \eqref{gammacond}
  and that the sequence $\{(y^t,z^t,x^t)\}$ generated from \eqref{scheme} has a cluster point $(y^*,z^*,x^*)$. Suppose also that $\P_\gamma$ is a KL function.
  Then the whole sequence $\{(y^t,z^t,x^t)\}$ is convergent.
\end{theorem}

%{\color{red} TK: Do we do anything local?}

As we have seen from Theorems~\ref{thm1} and \ref{thm2}, our convergence analysis of the PR splitting method requires that the nonconvex objective function can be decomposed as $f+g$ where $f$ is strongly convex. It should be noted that if the strong convexity assumption on $f$ is dropped, then the sequence generated is not necessarily converging to/clustering at a stationary point even when $g$ is also convex. On the other hand, in the next section, we will demonstrate how the method can be suitably applied to minimizing a coercive function $F+G$, where $G$ is a proper closed function and $F$ is convex Lipschitz differentiable but {\em not} necessarily strongly convex.

\section{Peaceman-Rachford splitting methods for nonconvex problems with non-strongly convex decomposition}\label{sec4}
  In many applications, the underlying optimization problem can be formulated as
  \begin{equation}\label{appliedP}
    \begin{array}{rl}
      \min\limits_u & F(u) + G(u)
    \end{array}
  \end{equation}
  where $F + G$ is {\em coercive}, $F$ is a convex smooth function with a Lipschitz continuous gradient whose modulus is at most $L_F > 0$, and $G$ is a proper and closed function with a nonempty proximal mapping ${\rm prox}_{\tau G}(z)$ for any $z$ and any $\tau > 0$. For example, when $F$ is the least squares loss function for linear regression and $G$ is the indicator function of the $\ell_1$ norm ball, the problem \eqref{appliedP} reduces to the LASSO \cite{Tib96}. This and various related (possibly nonconvex) models have been studied extensively in the statistical literature; see, for example, \cite{AtBoSv13,BvdBSC13,CaTa05,KnF00,FanJ}.
  We will also provide more concrete examples and simulation results later in Section \ref{sec5}.

  In view of the structure of \eqref{appliedP}, a natural way of applying a splitting method would be to set $f(y) = F(y)$ and $g(z) = G(z)$.
  However, since this choice of $f$ is not strongly convex, our convergence theory in Section~\ref{sec3} cannot be applied to deducing convergence of the resulting PR splitting method.

  Thus, we consider an alternative way of splitting the objective in order to obtain a strongly convex $f$. To this end, we start by fixing any $\alpha > 0$ and defining $f(y) = F(y) + \frac\alpha2\|y\|^2$, $g(z) = G(z) - \frac\alpha2\|z\|^2$. Then $\nabla f$ is Lipschitz continuous with a modulus at most $L = L_F + \alpha$, and $f$ is strongly convex with modulus at least $\sigma = \alpha$. Thus, one only needs to pick $\alpha > 2L_F$ so that $3\sigma > 2L$. Let $\alpha = \beta L_F$ for some $\beta > 2$. Then the upper bound of $\gamma$ in \eqref{gammacond} is given by
  \begin{equation*}
    \frac{\alpha  -2L_F}{(L_F + \alpha)^2} = \frac{\beta -2}{(\beta + 1)^2L_F}.
  \end{equation*}
   Consequently, if we set
  \[
  f(y) = F(y) + \frac{\beta L_F}2\|y\|^2 \ {\rm and}\ g(z) = G(z) - \frac{\beta L_F}2\|z\|^2,
  \]
  then we can pick $0 < \gamma < \frac{\beta -2}{(\beta + 1)^2L_F}$.\footnote{One natural choice of $\beta$ is to set $\beta = 5$ so that $\max_{\beta > 2}\frac{\beta -2}{(\beta + 1)^2L_F} = \frac1{12L_F}$ is attained. However, we discover in our numerical experiments that a smaller $\beta > 2$ coupled with a suitable heuristic for updating $\gamma$ leads to faster convergence.} Moreover, for this choice of $\gamma$, the Assumption~\ref{assumg} is satisfied for the above choice of $g$. Hence, it follows from Theorem~\ref{thm2} that the sequence generated by applying the PR splitting method to this pair of $f$ and $g$ is bounded, and then any cluster point gives a stationary point of \eqref{appliedP}, according to Theorem~\ref{thm1}.
  For concreteness and easy reference for our subsequent discussion, we present this algorithm explicitly below:\\

  \fbox{\parbox{5.7 in}{
\begin{description}
\item {\bf PR splitting method for \eqref{appliedP}}

\item[Step 0.] Input $x^0$, $\beta > 2$ and $\gamma \in \left(0,\frac{\beta -2}{(\beta + 1)^2L_F}\right)$.

\item[Step 1.] Set
\begin{equation}\label{scheme0}
\left\{
\begin{split}
&y^{t+1}=\argmin_{y} \left\{F(y) + \frac{\beta L_F}{2}\|y\|^2 + \frac{1}{2\gamma}\|y - x^t\|^2\right\}, \\
&z^{t+1}\in\Argmin_{z} \left\{G(z) - \frac{\beta L_F}{2}\|z\|^2 + \frac{1}{2\gamma}\|2y^{t+1} - x^t - z\|^2\right\},\\
&x^{t+1}=x^t+2(z^{t+1}- y^{t+1}).
\end{split}
\right.
\end{equation}

\item[Step 2.] If a termination criterion is not met, go to Step 1.
\end{description}
}}
\vspace{.1 in}

  To the best of our knowledge, the global convergence of the sequence generated from \eqref{scheme0} is new, which we summarize below for concreteness.

  \begin{corollary}\label{cor:1}
    Consider optimization problem \eqref{appliedP} and let $\{(y^t,z^t,x^t)\}$ be the sequence generated from \eqref{scheme0}. Then the sequence is bounded, and any cluster point $(\bar y,\bar z,\bar x)$ would satisfy $\bar y = \bar z$, and $\bar z$ is a stationary point of \eqref{appliedP}, that is,
    \[
      0\in \nabla F(\bar z) + \partial G(\bar z).
\]
  \end{corollary}
  \begin{proof}
    We first note that since \eqref{scheme0} is just \eqref{scheme} applied to $f(y) = F(y) + \frac{\beta L_F}{2}\|y\|^2$ and $g(z) = G(z) - \frac{\beta L_F}{2}\|z\|^2$, we obtain immediately from the above discussion and Theorem~\ref{thm1} that $\bar y = \bar z$ and $\bar z$ is a stationary point of \eqref{appliedP} for any cluster point $(\bar y,\bar z,\bar x)$. In addition, the objective function $f + g = F + G$ is coercive by assumption. The boundedness of the sequence $\{(y^t,z^t,x^t)\}$ now follows from Theorem~\ref{thm2}. This completes the proof.
\end{proof}

\subsection{Peaceman-Rachford splitting method for convex problems}

In this subsection, we suppose in addition that the $G$ in \eqref{appliedP} is also convex. Hence, \eqref{appliedP} is a convex problem. We first establish the following global (ergodic) complexity result for the sequence generated from \eqref{scheme0}. Similar kinds of complexity results have also been established for other primal-dual methods for convex optimization problems; see, for example, \cite[Theorem~2]{WangBan13}.
We would like to emphasize that the PR splitting method we discuss here is different from the classical PR splitting method in the literature: we split the convex objective $F+G$ into the sum of a strongly convex function $f$ and a possibly {\em nonconvex} function $g$, while the classical PR splitting method only admits splitting into a sum of convex functions.

\begin{theorem}[{{\bf Global iteration complexity under convexity}}]
Consider optimization problem \eqref{appliedP} with $G$ being convex. Let $\{(y^t,z^t,x^t)\}$ be the sequence generated from \eqref{scheme0} and $(\bar y,\bar z,\bar x)$ be any cluster point of this sequence. Then, $\bar y = \bar z$ and $\bar z$ is a solution of \eqref{appliedP}. Moreover, for any $N \ge 1$, we have
    \begin{equation}\label{complexity}
    F(\bar z^N) + G(\bar z^N) - F(\bar z) - G(\bar z) \le \frac1{8\beta\gamma NL_F}\left(\frac{1}{\gamma} - \beta L_F\right)\|x^0 - \bar x\|^2,
    \end{equation}
    where $\bar z^N := \frac{1}{N}\sum_{t=1}^N z^t$ and
    \[
     \min_{0 \le t \le N}\{\|x^{t+1} -x^{t}\|\}=o(\frac{1}{\sqrt{N}}).
    \]
\end{theorem}
    \begin{proof}
    Since \eqref{appliedP} is convex, we conclude that $\bar z$ is actually optimal.
    We now establish the inequality \eqref{complexity}. First, from the first-order optimality conditions for the $y$ and $z$-updates in \eqref{scheme0}, we have
    \begin{equation}\label{1storderop}
      \begin{split}
        -\left(\beta L_F + \frac{1}{\gamma}\right)y^{t+1} + \frac{1}{\gamma}x^t & = \nabla F(y^{t+1}),\\
        \left(\beta L_F - \frac{1}{\gamma}\right) z^{t+1} - \frac{1}{\gamma}x^t + \frac{2}{\gamma}y^{t+1}& \in \partial G(z^{t+1}).
      \end{split}
    \end{equation}
    Moreover, it is not hard to see from the definition of cluster point and \eqref{cond1} that \eqref{1storderop} is also satisfied with $\bar x$ in place of $x^t$ and $(\bar y,\bar z)$ in place of $(y^{t+1},z^{t+1})$.
    Write $w^t_e = w^t - \bar w$ for $w = x$, $y$ or $z$ for notational simplicity. We have from \eqref{1storderop} (and its counterpart at $(\bar y,\bar z,\bar x)$) and the monotonicity of convex subdifferentials that
    \begin{equation*}
        \left\langle -\left(\beta L_F + \frac{1}{\gamma}\right)y^{t+1}_e + \frac{1}{\gamma}x^t_e,y^{t+1}_e\right\rangle  \ge 0,\ \
        \left\langle \left(\beta L_F - \frac{1}{\gamma}\right) z^{t+1}_e - \frac{1}{\gamma}x^t_e + \frac{2}{\gamma}y^{t+1}_e,z^{t+1}_e\right\rangle  \ge 0.
    \end{equation*}
    Summing these two relations and rearranging terms, we obtain that
    \begin{equation}\label{bd0}
      \langle x^t_e,y^{t+1} - z^{t+1}\rangle + 2\langle y_e^{t+1},z^{t+1}_e\rangle \ge (1 + \beta\gamma L_F)\|y^{t+1}_e\|^2 + (1 - \beta\gamma L_F)\|z^{t+1}_e\|^2.
    \end{equation}
    Next, observe that
    \begin{equation}\label{rel2}
    \begin{split}
      \langle x^t_e,y^{t+1} - z^{t+1}\rangle &= \frac12\langle x^t_e,x^t - x^{t+1}\rangle = \frac14 (\|x^t_e\|^2 + \|x^t - x^{t+1}\|^2 - \|x^{t+1}_e\|^2)\\
      & = \frac14(\|x^t_e\|^2 - \|x^{t+1}_e\|^2) + \|z^{t+1} - y^{t+1}\|^2\\
      & = \frac14(\|x^t_e\|^2 - \|x^{t+1}_e\|^2) + \|z_e^{t+1}\| + \|y_e^{t+1}\|^2 - 2\langle y_e^{t+1},z^{t+1}_e\rangle,
    \end{split}
    \end{equation}
    where the first and third equalities follow from the third relation in \eqref{scheme0}, the second equality follows from the elementary relation $\langle u,v\rangle = \frac{1}{2}(\|u\|^2 + \|v\|^2 - \|u - v\|^2)$ as applied to $u = x_e^t$ and $v = x^t - x^{t+1}$. Combining \eqref{rel2} with \eqref{bd0}, we see further that
    \begin{equation}\label{relation1}
      \frac14 \|x_e^t\|^2 - \frac14 \|x_e^{t+1}\|^2 \ge \beta\gamma L_F(\|y^{t+1}_e\|^2 - \|z^{t+1}_e\|^2)
    \end{equation}

    Next, using the fact that $\nabla F$ is Lipschitz continuous with modulus at most $L_F$, we have
    \begin{equation}\label{LF}
        F(z^{t+1}) \le F(y^{t+1}) + \langle \nabla F(y^{t+1}), z^{t+1} - y^{t+1}\rangle + \frac{L_F}{2}\|z^{t+1} - y^{t+1}\|^2.
    \end{equation}
    From this we see further that
    \begin{equation}\label{keys}
      \begin{split}
        & F(z^{t+1}) + G(z^{t+1}) - F(\bar z) - G(\bar z)\\
        & \le F(y^{t+1}) - F(\bar y) + G(z^{t+1}) - G(\bar z) + \langle \nabla F(y^{t+1}), z^{t+1} - y^{t+1}\rangle + \frac{L_F}{2}\|z^{t+1} - y^{t+1}\|^2\\
        & \le \langle \nabla F(y^{t+1}),y^{t+1}_e\rangle + \left\langle \left(\beta L_F - \frac{1}{\gamma}\right) z^{t+1} - \frac{1}{\gamma}x^t + \frac{2}{\gamma}y^{t+1},z^{t+1}_e\right\rangle\\
        & \ \ \ + \langle \nabla F(y^{t+1}), z^{t+1} - y^{t+1}\rangle + \frac{L_F}{2}\|z^{t+1} - y^{t+1}\|^2\\
        & = \left\langle \nabla F(y^{t+1}) + \left(\beta L_F - \frac{1}{\gamma}\right) z^{t+1} - \frac{1}{\gamma}x^t + \frac{2}{\gamma}y^{t+1},z^{t+1}_e\right\rangle + \frac{L_F}{2}\|z^{t+1} - y^{t+1}\|^2\\
        & = \left\langle -\left(\beta L_F - \frac{1}{\gamma}\right)y^{t+1} + \left(\beta L_F - \frac{1}{\gamma}\right) z^{t+1},z^{t+1}_e\right\rangle + \frac{L_F}{2}\|z^{t+1} - y^{t+1}\|^2\\
        & = \left(\frac{1}{\gamma} - \beta L_F\right)\langle y^{t+1} - z^{t+1},z^{t+1}_e\rangle  + \frac{L_F}{2}\|z^{t+1} - y^{t+1}\|^2\\
        & = \frac12\left(\frac{1}{\gamma} - \beta L_F\right) (\|y^{t+1}_e\|^2 - \|z^{t+1}_e\|^2) + \frac12\left((1 + \beta)L_F - \frac{1}{\gamma}\right)\|z^{t+1} - y^{t+1}\|^2\\
        & \le \frac12\left(\frac{1}{\gamma} - \beta L_F\right) (\|y^{t+1}_e\|^2 - \|z^{t+1}_e\|^2) \le \frac1{8\beta\gamma L_F}\left(\frac{1}{\gamma} - \beta L_F\right)(\|x_e^t\|^2 - \|x_e^{t+1}\|^2),
      \end{split}
    \end{equation}
    where: the first inequality follows from \eqref{LF} and the fact that $\bar z = \bar y$; the second inequality follows from the subdifferential inequalities applied to $F$ and $G$ at the points $y^{t+1}$ and $z^{t+1}$ respectively, and also the second relation in \eqref{1storderop}; the second equality follows from the first relation in \eqref{1storderop}; the fourth equality follows from the elementary relation $\langle u,v\rangle = \frac{1}{2}(\|u + v\|^2 - \|u\|^2 - \|v\|^2)$ as applied to $u = z^{t+1}_e$ and $v = y^{t+1} - z^{t+1}$; the second last inequality follows from the fact that $0 < \gamma < \frac{\beta -2}{(\beta + 1)^2L_F}$ so that $(1 + \beta)L_F - \frac{1}{\gamma} < 0$, while the last inequality follows from \eqref{relation1}.

    Summing both sides of \eqref{keys} from $t=0$ to $N - 1 \ge 0$ and using the convexity of $F + G$, we have
    \begin{equation*}
    \begin{split}
      F(\bar z^N) + G(\bar z^N) - F(\bar z) - G(\bar z) & \le \frac{1}{N}\sum_{t=0}^{N-1}(F(z^{t+1}) + G(z^{t+1}) - F(\bar z) - G(\bar z))\\
      & \le \frac1{8\beta\gamma NL_F}\left(\frac{1}{\gamma} - \beta L_F\right)\|x^0 - \bar x\|^2,
    \end{split}
    \end{equation*}
    where $\bar z^N$ is defined in the statement of the theorem. This proves \eqref{complexity}.

    Finally, observe from the last equality in \eqref{keys} that for all $t \ge 1$
    \begin{equation*}
    \begin{split}
      0&\le F(z^{t+1}) + G(z^{t+1}) - F(\bar z) - G(\bar z) \\
      &\le \frac12\left(\frac{1}{\gamma} - \beta L_F\right) (\|y^{t+1}_e\|^2 - \|z^{t+1}_e\|^2) + \frac12\left((1 + \beta)L_F - \frac{1}{\gamma}\right)\|z^{t+1} - y^{t+1}\|^2,
    \end{split}
    \end{equation*}
    where the first inequality follows from the optimality of $\bar z$. Rearranging terms in the above relation, we see further that
    \begin{equation*}
      \left(\frac{1}{\gamma} - (1 + \beta)L_F\right)\|z^{t+1} - y^{t+1}\|^2 \le \left(\frac{1}{\gamma} - \beta L_F\right) (\|y^{t+1}_e\|^2 - \|z^{t+1}_e\|^2).
    \end{equation*}
    Using this relation and the definition of the $x$-update, we obtain
    \[
    \begin{split}
     \frac{1}{4}\sum_{t=0}^{N-1}\|x^{t+1}-x^t\|^2&=	\sum_{t=0}^{N-1}\|z^{t+1} - y^{t+1}\|^2 \le \frac{\gamma}{1 - (1 + \beta)\gamma L_F}\left(\frac{1}{\gamma} - \beta L_F\right)\sum_{t=0}^{N-1}(\|y^{t+1}_e\|^2 - \|z^{t+1}_e\|^2) \\
     &\le \frac1{4\beta L_F (1-(1 + \beta) \gamma L_F)}\left(\frac{1}{\gamma} - \beta L_F\right)\|x^0 - \bar x\|^2,
    \end{split}
    \]
    where the last inequality is due to \eqref{relation1}. Thus, $\sum_{t=0}^{+\infty}\|x^{t+1}-x^t\|^2<+\infty$ and so, $\sum_{t=N}^{2N-1} \|x^{t+1}-x^t\|^2 \rightarrow 0$ as $N \rightarrow \infty$. Now consider $\alpha_N:= \min_{0 \le t \le N}\{\|x^{t+1} -x^t\|^2\}$ for all $N \ge 0$. Then, we have $\alpha_{N+1} \le \alpha_N$ for all $N \ge 0$ and,
\[
N \, \alpha_{2N} \le \alpha_N + \ldots \alpha_{2N-1} \le \sum_{t=N}^{2N-1} \|x^{t+1}-x^t\|^2 \rightarrow 0.
\]
This implies that $\alpha_N = o(1/N)$. Therefore, the conclusion follows.
    This completes the proof.
  \end{proof}

Next, we show that the PR splitting method exhibits linear convergence in solving \eqref{appliedP} if $G$ is convex and $F+G$ is strongly convex. We note that,
for the classical PR splitting method, linear convergence under strongly convexity is known; see \cite[Remark 10 and Proposition 4]{LionsM79}. As explained before, here we are considering a different PR splitting method.

  \begin{proposition}{\bf (Linear convergence under strong convexity)}
    Consider optimization problem \eqref{appliedP} with $G$ being convex. Suppose that $F+G$ is indeed strongly convex. Let $\{(y^t,z^t,x^t)\}$ be the sequence generated from \eqref{scheme0}. Then
   $\{(y^t,z^t,x^t)\}$ converges linearly to $(\bar y, \bar z,\bar x)$ with  $\bar y = \bar z$ and $\bar z$ being the unique optimal solution for \eqref{appliedP}, i.e., there exist $M>0$ and $r \in (0,1)$ such that for all  $t \ge 1$,
    $$\max\{\|y^{t}-\bar y\|^2, \|z^{t}-\bar z\|^2, \|x^{t}-\bar x\|^2\} \le M \, r^t.$$
\end{proposition}
  \begin{proof}
  Let $(\bar y,\bar z,\bar x)$ be any cluster point of the sequence $\{(y^t,z^t,x^t)\}$. As before, we write $w^t_e = w^t - \bar w$ for $w = x$, $y$ or $z$ for notational simplicity. From the preceding theorem  $\bar y = \bar z$ and $\bar z$ is optimal for \eqref{appliedP}. Note that $F+G$ is strongly convex. Hence, the optimal solution of \eqref{appliedP} exists and is unique. Consequently, the whole sequence $\{(y^t,z^t)\}$ converges to the {\em unique} limit $(\bar z,\bar z)$, where $\bar z$ is the unique solution of \eqref{appliedP}. From this and \eqref{1storderop} one can deduce that $\{x^t\}$ is also convergent, and hence, converges to $\bar x$. We next establish linear convergence.

  Denote the strong convexity modulus of $F+G$ by $\sigma_1$.
From \eqref{keys}, the strong convexity of $F+G$ and the fact that $\bar z$ is the solution of \eqref{appliedP}, we see that for all $t \ge 1$,
\begin{equation}\label{eq:01}
\frac{\sigma_1}{2}\|z_e^{t+1}\|^2 \le F(z^{t+1}) + G(z^{t+1}) - F(\bar z) - G(\bar z) \le C(\|x_e^t\|^2 - \|x_e^{t+1}\|^2),
\end{equation}
where $C:=\frac1{8\beta\gamma L_F}\left(\frac{1}{\gamma} - \beta L_F\right)$. Moreover, from the last inequality in \eqref{keys}, we have for all $t \ge 1$,
\[
 C_1 (\|y_e^{t+1}\|^2 - \|z_e^{t+1}\|^2) \le  C(\|x_e^t\|^2 - \|x_e^{t+1}\|^2),
\]
where $C_1=\frac12\left(\frac{1}{\gamma} - \beta L_F\right)$. It then follows that
\[
\|y_e^{t+1}\|^2 -\frac{C}{C_1}(\|x_e^t\|^2 - \|x_e^{t+1}\|^2)\le \|z_e^{t+1}\|^2.
\]
This together with (\ref{eq:01}) gives us that for all $t \ge 1$,
\begin{equation} \label{eq:pp1}
  \|y_e^{t+1}\|^2 \le \left(\frac{2C}{\sigma_1} + \frac{C}{C_1}\right) (\|x_e^{t}\|^2 - \|x_e^{t+1}\|^2).
\end{equation}
On the other hand, note from the first relation in \eqref{1storderop} that
\[
-\left(\beta L_F + \frac{1}{\gamma}\right)y_e^{t+1} + \frac{1}{\gamma}x_e^{t}  = \nabla F(y^{t+1}) - \nabla F(\bar y).
\]
This together with the Lipschitz continuity of $\nabla F$ implies that
\[
-\left(\beta L_F + \frac{1}{\gamma}\right)\|y_e^{t+1}\| + \frac{1}{\gamma}\|x_e^{t}\| \le \|\nabla F(y^{t+1}) - \nabla F(\bar y)\| \le L_F \|y_e^{t+1}\|
\]
and consequently,
$\|x_e^{t}\| \le ((1 + \beta) \gamma L_F + 1) \|y_e^{t+1}\|$.
Thus, we obtain that, for all $t \ge 1$
\[
\frac{1}{((1 + \beta) \gamma L_F + 1)^2}\|x_e^t\|^2 \le \|y_e^{t+1}\|^2 \le \left(\frac{2C}{\sigma_1} + \frac{C}{C_1}\right) (\|x_e^t\|^2 - \|x_e^{t+1}\|^2).
\]
This shows that there exists $r \in (0,1)$ such that
$$\|x_e^{t+1}\|^2 \le  r \|x_e^t\|^2 \mbox{ for all } t \ge 1.$$
It follows that
\[
\|x_e^{t}\|^2 \le  \|x^0-\bar x\|^2 \, r^{t} \mbox{ for all } t \ge 1.
\]
Moreover, from (\ref{eq:01}) and (\ref{eq:pp1}), this further yields that, for all $t \ge 1$,
\[
\|z_e^{t+1}\|^2 \le \frac{2C }{\sigma_1} \|x_e^t\|^2 \le  \frac{2C \|x^0-\bar x\|^2}{\sigma_1} r^{t} .
\]
and
\[
  \|y_e^{t+1}\|^2
  \le \left(\frac{2C}{\sigma_1} + \frac{C}{C_1}\right) \|x_e^{t}\|^2 \le \left(\frac{2C}{\sigma_1} + \frac{C}{C_1}\right) \|x^0-\bar x\|^2  \, r^{t}.
\]
Therefore, the conclusion follows.
  \end{proof}

\section{Applications}\label{sec5}
  In this section, we apply the PR splitting method \eqref{scheme0} to solving two important class of nonconvex optimization problems: constrained least squares
  problem and feasibility problems, based on our discussion in Section~\ref{sec4}.
  %We next discuss the computation of the proximal mapping of $\gamma g$ and that of $\gamma f$ for $0 < \gamma < \frac1{12L_F}$ for some specific applications.

\paragraph{Constrained least squares problems.} A common type of problems that arises in the area of statistics and machine learning is the following constrained least squares problem:
\begin{equation}\label{LSP}
    \begin{array}{rl}
      \min\limits_{u\in D} & \frac12 \|Au - b\|^2,
    \end{array}
\end{equation}
where $A$ is a linear map, $b$ is a vector of suitable dimension, and $D$ is a nonempty {\em compact} set that is not necessarily convex. See \cite{KBC12,Tib96} for concrete examples of \eqref{LSP}.

The classical PR splitting method applied to \eqref{LSP} does not have a convergence guarantee. As an alternative, as discussed in Section~\ref{sec4}, we can set $f(y) = \frac{1}{2}\|Ay - b\|^2 + \frac{\beta\lambda_{\max}(A^TA)}{2}\|y\|^2$ and $g(z) = \delta_D(z) - \frac{\beta\lambda_{\max}(A^TA)}{2}\|z\|^2$ and apply the PR splitting method accordingly.

We next discuss computation of the proximal mappings. We start with the proximal mapping of $\gamma g$.
  From the definition, for each $w$, the proximal mapping gives the set of minimizers of
  \[
  \min_{z\in D}\left\{-\frac{\beta\lambda_{\max}(A^TA)}{2}\|z\|^2 + \frac{1}{2\gamma}\|z - w\|^2\right\}.
  \]
  It is clear that this set is given by $P_D\left(\frac{w}{1 - \beta\lambda_{\max}(A^TA)\gamma}\right)$ since $\gamma < \frac{1}{\beta\lambda_{\max}(A^TA)}$.
On the other hand, to compute the proximal mapping for $\gamma f$, we consider the following optimization problem for each $w$
\begin{equation*}
    \min_{y}\left\{\frac12 \|Ay - b\|^2 + \frac{\beta\lambda_{\max}(A^TA)}2 \|y\|^2 + \frac{1}{2\gamma}\|y - w\|^2\right\},
\end{equation*}
whose unique minimizer is given by
\[
  y = [(\beta\gamma\lambda_{\max}(A^TA) + 1)I + \gamma A^TA]^{-1}(w + \gamma A^Tb).
\]
Thus, the PR splitting method for \eqref{LSP} can be stated as follows:\\ \ \\
\fbox{\parbox{5.7 in}{
\begin{description}
\item {\bf PR splitting method for \eqref{LSP}}

\item[Step 0.] Input $x^0$, $\beta > 2$ and $\gamma \in \left(0,\frac{\beta -2}{(\beta + 1)^2\lambda_{\max}(A^TA)}\right)$.

\item[Step 1.] Set
\begin{equation}\label{scheme2}
\left\{
\begin{split}
&y^{t+1} = [(\beta\gamma\lambda_{\max}(A^TA) + 1)I + \gamma A^TA]^{-1}(x^t + \gamma A^Tb), \\
&z^{t+1}\in P_D\left(\frac{2y^{t+1} - x^t}{1 - \beta\lambda_{\max}(A^TA)\gamma}\right),\\
&x^{t+1}=x^t+2(z^{t+1}- y^{t+1}).
\end{split}
\right.
\end{equation}

\item[Step 2.] If a termination criterion is not met, go to Step 1.
\end{description}
}}
\vspace{.1 in}

As a consequence of Corollary \ref{cor:1}, we see that Algorithm \eqref{scheme2} generates a bounded sequence such that any of its cluster point gives a stationary point of \eqref{LSP}. We note that this global convergence result of \eqref{scheme2} is new even when $D$ is convex.

To illustrate our proposed approach, we now test the PR splitting method \eqref{scheme2} on solving \eqref{LSP}.
We compare our algorithm against the DR splitting method in \cite{LiPong14_2}. Our initialization and termination criteria for both algorithms are the same as in \cite[Section~5]{LiPong14_2}; both algorithms are initialized at the origin and terminated when
\begin{equation}\label{term_criterion}
\frac{\max\{\|x^t - x^{t-1}\|,\|y^t - y^{t-1}\|,\|z^t - z^{t-1}\|\}}{\max\{\|x^{t-1}\|,\|y^{t-1}\|,\|z^{t-1}\|,1\}}< tol
\end{equation}
for some $tol > 0$.
Note that, in general, the upper bound of $\gamma$ in algorithm \eqref{scheme2} might be too small in practical computation. Thus, following a technique used in \cite[Section~5]{LiPong14_2} for the DR splitting method, we adopt a heuristic for PR splitting method in our numerical simulation, which combines algorithm \eqref{scheme2} with a specific update rule of the parameter $\gamma$. In particular, we set $\beta = 2.2$ and start with $\gamma = 0.93/(\beta\lambda_{\max}(A^TA))$. We then update $\gamma$ as
$\max\{\frac\gamma 2,0.9999\cdot\gamma_1\}$ whenever $\gamma > \gamma_1 := \frac{\beta - 2}{(\beta + 1)^2\lambda_{\max}(A^TA)}$ and the sequence satisfies either $\|y^t-y^{t-1}\| > \frac{1000}t$ or $\|y^t\| > 10^{10}$. Following a similar discussion as in \cite[Remark~4]{LiPong14_2}, one can show that this heuristic leads to a bounded sequence which clusters at a stationary point of \eqref{LSP}. On the other hand, for the DR splitting method, we use the same heuristics described in \cite[Section~5]{LiPong14_2} for updating $\gamma$ but we consider three different initial $\gamma$'s: $k\cdot\gamma_0$ for $k=10$, $30$ and $50$, with $\gamma_0 = (\sqrt{\frac32} - 1)/\lambda_{\max}(A^TA)$. These variants are denoted by $\rm DR_{10}$, $\rm DR_{30}$ and $\rm DR_{50}$, respectively.

In our first numerical experiment, we first randomly generate an $m\times n$ matrix $A$, a noise vector $\epsilon\in\R^m$, and also an $\hat x\in \R^r$ with $r = \lceil \frac m{10}\rceil$, all with i.i.d. standard Gaussian entries. We further scale each column of $A$ to have norm $1$. Next, we generate a random sparse vector $\tilde x\in \R^n$ by first setting $\tilde x = 0$ and then assigning randomly $r$ entries in $\tilde x$ to be $\hat x$. Finally, we set $b = A\tilde x + 0.01\cdot\epsilon$ and $D = \{x\in \R^n:\; \|x\|_0 \le r,\ \|x\|_\infty\le 10^6\}$; here $\|x\|_0$ denotes the cardinality of $x$ and $\|x\|_\infty$ is the $\ell_\infty$ norm of $x$.

We generate $50$ random instances as described above for each pair of $(m,n)$, where $m\in \{100, 200, 300, 400,500\}$ and $n\in \{4000,5000,6000\}$. Our results are reported in Table~\ref{t}, where we present the number of iterations and the function value at termination\footnote{We choose $tol = 10^{-8}$, and we report $\frac12\|Az^t - b\|^2$ for both methods.} averaged over the $50$ instances. One can observe that the PR splitting method is faster than the DR splitting methods for larger $m$. Besides, the function values obtained by the PR splitting method are usually comparable with ${\rm DR_{30}}$, worse than ${\rm DR_{50}}$ and better than ${\rm DR_{10}}$.

\begin{table}[h]
\caption{Comparing ${\rm DR_{10}}$, ${\rm DR_{30}}$, ${\rm DR_{50}}$ and PR splitting for constrained least squares problem on random instances.}\label{t} \normalsize
%\begin{sidewaystable}
\scriptsize
\begin{center}
%\vspace*{.2cm}
\begin{tabular}{|r|r||r|c||r|c||r|c||r|c|}  \hline
\multicolumn{2}{|c||}{Data} & \multicolumn{2}{c||}{${\rm DR_{10}}$} & \multicolumn{2}{c||}{${\rm DR_{30}}$} & \multicolumn{2}{c||}{${\rm DR_{50}}$} & \multicolumn{2}{c|}{PR}
\\ \hline
%\cline{1-10}
$m$ & $n$ & ${\rm iter}$ & ${\rm fval}$ & ${\rm iter}$ & ${\rm fval}$ & ${\rm iter}$ & ${\rm fval}$ & ${\rm iter}$ & ${\rm fval}$
\\ \hline
   100 &  4000 &     805 & 5.00e-01 &     225 & 2.67e-01 &    274 & 7.73e-02 &    324 & 3.17e-01  \\
   100 &  5000 &     962 & 6.43e-01 &     252 & 4.96e-01 &    291 & 2.06e-01 &    370 & 4.95e-01  \\
   100 &  6000 &    1137 & 6.18e-01 &     326 & 5.02e-01 &    301 & 2.53e-01 &    436 & 4.76e-01  \\
   200 &  4000 &     508 & 5.32e-01 &     172 & 4.74e-02 &    217 & 9.20e-03 &    185 & 7.59e-02  \\
   200 &  5000 &     624 & 5.78e-01 &     195 & 6.93e-02 &    234 & 9.10e-03 &    224 & 2.06e-01  \\
   200 &  6000 &     723 & 6.93e-01 &     220 & 1.60e-01 &    250 & 8.94e-03 &    281 & 1.77e-01  \\
   300 &  4000 &     415 & 1.41e-01 &     141 & 1.33e-02 &    184 & 1.31e-02 &    123 & 1.39e-02  \\
   300 &  5000 &     489 & 2.70e-01 &     154 & 1.39e-02 &    201 & 1.35e-02 &    150 & 1.42e-02  \\
   300 &  6000 &     567 & 5.20e-01 &     170 & 1.36e-02 &    215 & 1.32e-02 &    187 & 1.44e-02  \\
   400 &  4000 &     322 & 4.35e-02 &     124 & 1.78e-02 &    166 & 1.75e-02 &     91 & 1.79e-02  \\
   400 &  5000 &     406 & 9.08e-02 &     137 & 1.77e-02 &    179 & 1.75e-02 &    115 & 1.83e-02  \\
   400 &  6000 &     481 & 1.48e-01 &     148 & 1.82e-02 &    194 & 1.77e-02 &    140 & 1.85e-02  \\
   500 &  4000 &     258 & 2.53e-02 &     114 & 2.26e-02 &    160 & 2.23e-02 &     75 & 2.27e-02  \\
   500 &  5000 &     314 & 2.97e-02 &     124 & 2.20e-02 &    166 & 2.17e-02 &     92 & 2.22e-02  \\
   500 &  6000 &     406 & 4.05e-02 &     135 & 2.25e-02 &    178 & 2.22e-02 &    112 & 2.27e-02  \\
\hline
\end{tabular}
\end{center}
%\end{sidewaystable}

\end{table}

We also perform experiments using real data. We consider four sets of real data for the $A$ and $b$ used in \eqref{LSP}: leukemia data, lymph node status data, breast cancer prognosis data and colon tumor gene expression data. We use the leukemia data pre-processed in \cite{Yeung05}, that has $3501$ genes and $72$ samples. The lymph node status data we use are pre-processed in \cite{Dobra09}, with 4514 genes and 148 samples. The breast cancer prognosis data we use are pre-processed in \cite{Yeung05}, containing 4919 genes and 76 samples. Finally, we use the data pre-processed in \cite{Golub99} with 2000 genes and 62 samples for the colon tumor gene expression data.

Similar to \cite[Section~3.3]{Lu12}, for all the data, we first standardize $A$ and $b$ to make each column have mean 0 and variance 1, and then scale the columns of $A$ to have unit norm. For the $A$ and $b$ thus constructed, we solve \eqref{LSP} with $D = \{x\in \R^n:\; \|x\|_0 \le r,\ \|x\|_\infty\le 10^6\}$ for $r = 10$, $20$, $30$ by the PR splitting method \eqref{scheme2} and compare it with ${\rm DR_{10}}$, ${\rm DR_{30}}$ and ${\rm DR_{50}}$.
Our numerical results are presented in Table~\ref{real},\footnote{We choose $tol = 10^{-5}$, and we report $\frac12\|Az^t - b\|^2$ for both methods.} where one can see that PR is slower than ${\rm DR_{50}}$ and faster than ${\rm DR_{10}}$. Moreover, it usually outperforms ${\rm DR}_{30}$ in terms of function values, and its speed is comparable with ${\rm DR_{30}}$ for the Breast and the Colon data.

\begin{table}[h]
\caption{Comparing ${\rm DR_{10}}$, ${\rm DR_{30}}$, ${\rm DR_{50}}$ and PR splitting on real data.}\label{real} \normalsize
%\begin{sidewaystable}
\scriptsize
\begin{center}
%\vspace*{.2cm}
\begin{tabular}{|r||c||r|c||r|c||r|c||r|c|}  \hline
 \multirow{2}{*}{Data} & \multirow{2}{*}{$r$}  & \multicolumn{2}{c||}{${\rm DR_{10}}$} & \multicolumn{2}{c||}{${\rm DR_{30}}$} & \multicolumn{2}{c||}{${\rm DR_{50}}$} & \multicolumn{2}{c|}{PR}
\\ %\hline
\cline{3-10}
 & & ${\rm iter}$ & ${\rm fval}$ & ${\rm iter}$ & ${\rm fval}$ & ${\rm iter}$ & ${\rm fval}$ & ${\rm iter}$ & ${\rm fval}$
\\
\hline
          & 10 &  8242 & 2.40e+00 &  1805 & 3.92e+00 &  1229 & 3.92e+00 &  3461 & 2.47e+00\\
 Leukemia & 20 &  7890 & 2.32e+00 &  3727 & 6.09e-01 &  3065 & 5.81e-01 &  6608 & 3.05e-01\\
          & 30 & 12530 & 2.24e-01 &  5011 & 3.01e-01 &  2988 & 1.47e-01 &  8265 & 1.20e-01\\ \hline
          & 10 &  1345 & 2.93e+01 &   758 & 2.90e+01 &   496 & 2.90e+01 &  1297 & 2.76e+01\\
   Lymph  & 20 &  5912 & 2.26e+01 &  1910 & 1.91e+01 &   895 & 1.73e+01 &  2529 & 1.84e+01\\
          & 30 &  9354 & 7.91e+00 &  1883 & 1.34e+01 &   939 & 1.44e+01 &  2089 & 8.27e+00\\ \hline
          & 10 &  2338 & 1.28e+01 &  2705 & 9.33e+00 &  1095 & 8.40e+00 &  1656 & 1.33e+01\\
   Breast & 20 & 14359 & 2.90e+00 &  2345 & 3.53e+00 &  2824 & 4.11e+00 &  2906 & 2.81e+00\\
          & 30 &  9905 & 6.96e-01 &  5162 & 1.33e+00 &  3802 & 7.50e-01 &  8241 & 9.58e-01\\ \hline
          & 10 &  7072 & 8.08e+00 &  4313 & 8.08e+00 &  3352 & 8.08e+00 &  4463 & 8.08e+00\\
   Colon  & 20 & 14393 & 3.20e+00 &  7011 & 1.95e+00 &  9798 & 2.29e+00 &  6187 & 1.89e+00\\
          & 30 & 18361 & 7.17e-01 &  8952 & 6.45e-01 &  4922 & 7.26e-01 & 10937 & 1.33e+00\\ \hline
\end{tabular}
\end{center}
%\end{sidewaystable}

\end{table}

\paragraph{Feasibility problems.} Another important problem in optimization is the feasibility problem \cite{AtBoSv13,Bauschke_SIAM_review,Jon1993,Li,Luke1}. We consider the following simple version: finding a point in the intersection of a nonempty closed convex set $C$ and a nonempty {\em compact} set $D$. It is well known that this problem can be modeled via \eqref{appliedP} by setting $F(u) = \frac12 d_C^2(u)$ and $G(u) = \delta_D(u)$; see, for example, \cite{Luke08}. For this choice of $F$, we have $L_F = 1$.

As before, it can be shown that the proximal mapping of $\gamma g$ is given by $P_D\left(\frac{w}{1 - \beta\gamma}\right)$ since $\gamma < \frac1\beta$.
%  We now compute the proximal mapping for $\gamma g$, where our $g(z) = \delta_D(z) - \frac{5}{2}\|z\|^2$.
%  From the definition, for each $w$, the proximal mapping gives the set of minimizers of
%  \[
%  \min_{z\in D}\left\{-\frac{5}{2}\|z\|^2 + \frac{1}{2\gamma}\|z - w\|^2\right\}.
%  \]
%  It is clear that a minimizer is given by ${\rm Proj}_D\left(\frac{w}{1 - 5\gamma}\right)$.
  We next compute the proximal mapping for $\gamma f$ in this case. From the definition, for each $w$, we consider the following optimization problem
  \begin{equation}\label{prox_f}
  \begin{split}
    v:=\min_{y}\left\{\frac12 d_C^2(y) + \frac{\beta}{2} \|y\|^2 + \frac{1}{2\gamma}\|y - w\|^2\right\} & = \min_{u\in C}\min_{y}\left\{\frac12 \|y - u\|^2 + \frac{\beta}{2} \|y\|^2 + \frac{1}{2\gamma}\|y - w\|^2\right\}.
  \end{split}
  \end{equation}
  Notice that the inner minimization on the right hand side is attained at
  \begin{equation}\label{y:expression}
  y = \frac{\gamma u + w}{(1 + \beta)\gamma + 1}.
  \end{equation}
  Plugging \eqref{y:expression} back into the \eqref{prox_f}, we see further that
  \begin{equation}\label{prox_f2}
    v = \frac{1}{((1 + \beta)\gamma + 1)^2}\min_{u\in C}\left\{\frac12 \| (1 + \beta\gamma)u - w\|^2 + \frac{\beta}{2} \|\gamma u + w\|^2 + \frac{\gamma}{2}\| u -(1 + \beta) w\|^2\right\}.
  \end{equation}
  It is routine to show that the minimum in \eqref{prox_f2} is attained at
  \[
  u = P_C\left(\frac{w}{1 + \beta\gamma}\right).
  \]
  Combining this with \eqref{y:expression}, the proximal mapping of $\gamma f$ at $w$ is given by
  \[
  \frac{\gamma P_C\left(\frac{w}{1 + \beta\gamma}\right) + w}{(1 + \beta)\gamma + 1}.
  \]
  Thus, the PR splitting method for \eqref{appliedP} with $F(u) = \frac12 d_C^2(u)$ and $G(u) = \delta_D(u)$ can be described as follows:\\ \ \\
\fbox{\parbox{5.7 in}{
\begin{description}
\item {\bf PR splitting method for \eqref{appliedP} with $F(u) = \frac12 d_C^2(u)$ and $G(u) = \delta_D(u)$}

\item[Step 0.] Input $x^0$, $\beta > 2$ and $\gamma \in \left(0,\frac{\beta -2}{(\beta + 1)^2}\right)$.

\item[Step 1.] Set
\begin{equation}\label{scheme3}
\left\{
\begin{split}
&y^{t+1} = \frac{\gamma P_C\left(\frac{x^t}{1 + \beta\gamma}\right) + x^t}{(1 + \beta)\gamma + 1}, \\
&z^{t+1}\in P_D\left(\frac{2y^{t+1} - x^t}{1 - \beta\gamma}\right),\\
&x^{t+1}=x^t+2(z^{t+1}- y^{t+1}).
\end{split}
\right.
\end{equation}

\item[Step 2.] If a termination criterion is not met, go to Step 1.
\end{description}
}}
\vspace{.1 in}

Similarly, as an immediate consequence of Corollary \ref{cor:1}, we see that Algorithm \eqref{scheme3} generates a bounded sequence such that any of its cluster point gives a stationary point of \eqref{appliedP}. We would like to point out that this global convergence result of \eqref{scheme3} is new even when $D$ is also convex.

As an illustration of our proposed approach, we now test the PR splitting method \eqref{scheme3} on solving \eqref{appliedP} with $F(u) = \frac12 d_C^2(u)$ and $G(u) = \delta_D(u)$ via MATLAB experiments.
We again benchmark our algorithm against the DR splitting method in \cite{LiPong14_2}. Both algorithms are initialized at the origin and terminated when \eqref{term_criterion} is satisfied with $tol = 10^{-8}$.
Also, as in the previous subsection, we adopt a heuristic for updating $\gamma$ following the technique used in \cite[Section~5]{LiPong14_2}. Specifically, for the PR splitting method \eqref{scheme3}, we set $\beta = 2.2$ and start with $\gamma = 0.93/\beta$  and update $\gamma$ as
$\max\{\frac\gamma 2,0.9999\cdot\gamma_1\}$ whenever $\gamma > \gamma_1 := \frac{\beta - 2}{(\beta + 1)^2}$, and the sequence satisfies either $\|y^t-y^{t-1}\| > \frac{1000}t$ or $\|y^t\| > 10^{10}$. Following a similar discussion as in \cite[Remark~4]{LiPong14_2}, this heuristic can be shown to give a bounded sequence that clusters at a stationary point of \eqref{appliedP}. On the other hand, for the DR splitting method, we adopt the same heuristics described in \cite[Section~5]{LiPong14_2} for updating $\gamma$ but we consider three different initial $\gamma$'s: $k\cdot\gamma_0$ for $k=50$, $100$ and $150$, with $\gamma_0 := \sqrt{\frac32} - 1$. These variants are denoted by $\rm DR_{50}$, $\rm DR_{100}$ and $\rm DR_{150}$, respectively.

As in \cite[Section~5]{LiPong14_2}, we consider the problem of finding an $r$-sparse solution of a randomly generated linear system $Ax = b$.
To be concrete, we set $C = \{x\in \R^n:\; Ax = b\}$ and $D = \{x\in \R^n:\; \|x\|_0 \le r,\ \|x\|_\infty\le 10^6\}$; here $\|x\|_0$ denotes the cardinality of $x$ and $\|x\|_\infty$ is the $\ell_\infty$ norm of $x$.
For the set $C$, we first generate an $m\times n$ matrix $A$ and an $\hat x\in \R^r$ with $r = \lceil \frac m{5}\rceil$, both with i.i.d. standard Gaussian entries. We then set $\tilde x$ to be the $n$-dimensional zero vector and randomly assign $r$ entries in $\tilde x$ to be $\hat x$. We further project this $\tilde x$ onto $[-10^6,10^6]^n$ so that $\tilde x\in D$. Finally, we set $b = A\tilde x$. Consequently, the intersection $C\cap D$ is nonempty for the instance generated because it contains $\tilde x$. In particular, this means that the globally optimal value of $\min_u\{\frac12d_C^2(u): u\in D\}$ is zero.

In our experiments, we generate $50$ random instances as described above for each pair of $(m,n)$, where $m\in \{100, 200, 300, 400,500\}$ and $n\in \{4000,5000,6000\}$. We report our results in Tables~\ref{t1} and \ref{t2}, where we present the number of iterations averaged over the $50$ instances, the largest and smallest function values at termination,\footnote{For both methods, we report $\frac12 d_C^2(z^t)$.} and also the number of successes and failures in identifying a sparse solution of the linear system.\footnote{We declare  a failure if the function value at termination is above $10^{-6}$, and a success if the value is below $10^{-12}$.} We also present the average number of iterations for successful instances (${\rm iter_s}$) and failed instances (${\rm iter_f}$).

In Table~\ref{t1}, we compare our PR splitting method with ${\rm DR}_{150}$. One can observe that this version of DR splitting method outperforms the PR splitting method in terms of the solution quality in this setting. However, the PR splitting method is consistently faster and its performance becomes comparable with the DR splitting method for easier instances (larger $m$ and smaller $n/m$).

We also present in Table~\ref{t2} the numerical results for ${\rm DR}_{50}$ and ${\rm DR}_{100}$. One can see that the DR splitting method becomes faster (while still slower than the PR splitting method) for these two smaller initial $\gamma$, at the price of fewer successful instances.

\begin{table}[h] \normalsize
\caption{Comparing ${\rm DR}_{150}$ and PR splitting on random instances.}\label{t1}
%\begin{sidewaystable}
\scriptsize
\begin{center}
%\vspace*{.2cm}
\begin{tabular}{|r|r||r|c|c|r|r|r|r||r|c|c|r|r|r|r|}  \hline
\multicolumn{2}{|c||}{Data} & \multicolumn{7}{c||}{${\rm DR_{150}}$} & \multicolumn{7}{c|}{PR}
\\ \hline
%\cline{1-16}
$m$ & $n$ & ${\rm iter}$ & ${\rm fval}_{\max}$ & ${\rm fval}_{\min}$ & succ &  fail & $\rm iter_{s}$ & $\rm iter_{f}$ & ${\rm iter}$ & ${\rm fval}_{\max}$ & ${\rm fval}_{\min}$ & succ &  fail & $\rm iter_{s}$ & $\rm iter_{f}$
\\ \hline
   100 &  4000 &    2073 & 3e-02 & 1e-16 & 36 & 14 &   1861 &   2617 &     297 & 6e-02 & 4e-05 &  0 & 50 &       - &    297\\
   100 &  5000 &    2931 & 3e-02 & 1e-16 & 12 & 38 &   1842 &   3275 &     367 & 5e-02 & 3e-05 &  0 & 50 &       - &    367\\
   100 &  6000 &    2014 & 2e-02 & 2e-16 &  5 & 45 &   1891 &   2028 &     431 & 5e-02 & 8e-08 &  0 & 49 &       - &    423\\
   200 &  4000 &     833 & 7e-02 & 3e-16 & 49 &  1 &    825 &   1219 &     189 & 2e-01 & 1e-15 & 15 & 35 &     227 &    173\\
   200 &  5000 &     970 & 5e-02 & 2e-16 & 48 &  2 &    947 &   1528 &     230 & 1e-01 & 2e-15 & 11 & 39 &     297 &    211\\
   200 &  6000 &    1254 & 4e-02 & 3e-16 & 44 &  6 &   1193 &   1704 &     277 & 1e-01 & 3e-15 &  4 & 46 &     344 &    271\\
   300 &  4000 &     607 & 3e-15 & 2e-16 & 50 &  0 &    607 &      - &     132 & 3e-01 & 9e-16 & 38 & 12 &     138 &    111\\
   300 &  5000 &     705 & 3e-15 & 3e-16 & 50 &  0 &    705 &      - &     163 & 2e-01 & 1e-15 & 24 & 26 &     181 &    146\\
   300 &  6000 &     819 & 3e-15 & 4e-16 & 50 &  0 &    819 &      - &     204 & 2e-01 & 2e-15 & 16 & 34 &     241 &    187\\
   400 &  4000 &     523 & 3e-15 & 5e-17 & 50 &  0 &    523 &      - &      95 & 2e-01 & 8e-16 & 44 &  6 &      96 &     91\\
   400 &  5000 &     574 & 4e-15 & 2e-16 & 50 &  0 &    574 &      - &     125 & 3e-01 & 1e-15 & 43 &  7 &     127 &    114\\
   400 &  6000 &     655 & 4e-15 & 5e-16 & 50 &  0 &    655 &      - &     156 & 3e-01 & 2e-15 & 27 & 23 &     165 &    145\\
   500 &  4000 &     500 & 2e-16 & 7e-19 & 50 &  0 &    500 &      - &     106 & 2e-01 & 6e-16 & 49 &  1 &      64 &   2173\\
   500 &  5000 &     521 & 1e-15 & 4e-17 & 50 &  0 &    521 &      - &      91 & 3e-01 & 1e-15 & 47 &  3 &      91 &     87\\
   500 &  6000 &     560 & 4e-15 & 4e-16 & 50 &  0 &    560 &      - &     123 & 3e-01 & 1e-15 & 47 &  3 &     124 &    108\\
\hline
\end{tabular}
\end{center}
%\end{sidewaystable}
\end{table}

\begin{table}[h] \normalsize
\caption{Computational results for ${\rm DR_{50}}$ and ${\rm DR_{100}}$.}\label{t2}
%\begin{sidewaystable}
\scriptsize
\begin{center}
%\vspace*{.2cm}
\begin{tabular}{|r|r||r|c|c|r|r|r|r||r|c|c|r|r|r|r|}  \hline
\multicolumn{2}{|c||}{Data} & \multicolumn{7}{c||}{${\rm DR_{50}}$} & \multicolumn{7}{c|}{${\rm DR_{100}}$}
\\ \hline
%\cline{1-16}
$m$ & $n$ & ${\rm iter}$ & ${\rm fval}_{\max}$ & ${\rm fval}_{\min}$ & succ &  fail & $\rm iter_{s}$ & $\rm iter_{f}$ & ${\rm iter}$ & ${\rm fval}_{\max}$ & ${\rm fval}_{\min}$ & succ &  fail & $\rm iter_{s}$ & $\rm iter_{f}$
\\ \hline
   100 &  4000 &     336 & 4e-02 & 6e-16 &  1 & 49 &    423 &    334 &     854 & 2e-02 & 2e-16 &  5 & 45 &     716 &    870\\
   100 &  5000 &     345 & 4e-02 & 3e-16 &  1 & 49 &    423 &    343 &     681 & 2e-02 & 4e-16 &  2 & 48 &     683 &    681\\
   100 &  6000 &     349 & 3e-02 & 5e-03 &  0 & 50 &      - &    349 &     647 & 2e-02 & 3e-16 &  1 & 49 &     715 &    646\\
   200 &  4000 &     331 & 1e-01 & 4e-16 & 17 & 33 &    351 &    321 &     711 & 7e-02 & 8e-17 & 48 &  2 &     669 &   1728\\
   200 &  5000 &     332 & 8e-02 & 9e-16 &  3 & 47 &    357 &    330 &     983 & 5e-02 & 1e-16 & 44 &  6 &     864 &   1857\\
   200 &  6000 &     341 & 7e-02 & 5e-16 &  6 & 44 &    396 &    333 &    1186 & 4e-02 & 1e-16 & 24 & 26 &     802 &   1540\\
   300 &  4000 &     319 & 2e-01 & 1e-16 & 45 &  5 &    315 &    353 &     489 & 3e-15 & 4e-16 & 50 &  0 &     489 &      -\\
   300 &  5000 &     332 & 1e-01 & 5e-16 & 29 & 21 &    335 &    328 &     545 & 3e-15 & 4e-16 & 50 &  0 &     545 &      -\\
   300 &  6000 &     341 & 1e-01 & 6e-16 & 16 & 34 &    378 &    323 &     674 & 5e-02 & 3e-16 & 49 &  1 &     651 &   1799\\
   400 &  4000 &     271 & 3e-15 & 9e-16 & 50 &  0 &    271 &      - &     405 & 4e-15 & 2e-16 & 50 &  0 &     405 &      -\\
   400 &  5000 &     301 & 1e-01 & 8e-16 & 48 &  2 &    296 &    413 &     453 & 4e-15 & 5e-16 & 50 &  0 &     453 &      -\\
   400 &  6000 &     329 & 1e-01 & 5e-16 & 40 & 10 &    330 &    329 &     516 & 4e-15 & 5e-16 & 50 &  0 &     516 &      -\\
   500 &  4000 &     244 & 5e-15 & 2e-16 & 50 &  0 &    244 &      - &     363 & 3e-15 & 2e-16 & 50 &  0 &     363 &      -\\
   500 &  5000 &     269 & 4e-15 & 7e-16 & 50 &  0 &    269 &      - &     404 & 5e-15 & 3e-16 & 50 &  0 &     404 &      -\\
   500 &  6000 &     295 & 5e-15 & 4e-16 & 50 &  0 &    295 &      - &     442 & 5e-15 & 9e-16 & 50 &  0 &     442 &      -\\
\hline
\end{tabular}
\end{center}
%\end{sidewaystable}
\end{table}

\section{Concluding remarks}\label{sec6}

In this paper, we studied the applicability of the PR splitting method for solving nonconvex optimization problems. We established global convergence of the method when applied to
minimizing the sum of a strongly convex Lipschitz differentiable function $f$ and a proper closed function $g$, under suitable assumptions. Exploiting the {\em possible nonconvexity} of $g$, we showed how to suitably apply the PR splitting method to a large class of convex optimization problems whose objective function is not necessarily strongly convex. This significantly broadens the applicability of the PR splitting method to cover feasibility problems and many constrained least squares problems.

%\appendix
\section*{Appendix: Concrete numerical examples}
In this appendix, we provide some simple and concrete examples illustrating the different behaviors of the classical PR splitting method, the classical DR splitting method and our proposed PR splitting method \eqref{scheme0}.

The first example shows that,  even in the convex setting, the classical PR splitting method can be faster than the classical DR splitting method, and our proposed PR method can outperform the classical DR method
for some particular choice of the parameter $\gamma$. The second example on nonconvex feasibility problem shows that the classical PR method can {\em diverge} while our
proposed PR method {\em converges linearly} to a solution for the feasibility problem.
\begin{example}{\bf (Classical DR splitting method vs classical/proposed PR method)} \label{ex:N1}
Consider $f(x)=\|x\|^2$ and $g(x)=0$ for all $x \in \R^n$. Then, a direct verification shows that, for any $\gamma>0$,
  \begin{equation*}%\label{proxh0}
  {\rm prox}_{\gamma f}(z) = \argmin_u\left\{\gamma \|u\|^2 + \frac12 \|u - z\|^2\right\}=\frac{z}{2\gamma+1}
  \end{equation*}
  and
  \begin{equation*}%\label{proxh0}
  {\rm prox}_{\gamma g}(z) = \argmin_u\left\{ \frac12 \|u - z\|^2\right\}=z.
  \end{equation*}
  Thus, the classical DR method reads
     \[
  x^{t+1}=\frac{I+(2{\rm prox}_{\gamma g} - I)\circ(2{\rm prox}_{\gamma f} - I)}{2}(x^t)=\frac{1}{2\gamma+1} x^t =\cdots= \left(\frac{1}{2\gamma+1}\right)^{t+1} x^0,
  \]
  while the classical PR method reads
  \[
  x^{t+1}=(2{\rm prox}_{\gamma g} - I)\circ(2{\rm prox}_{\gamma f} - I)(x^t)=\frac{1-2\gamma}{2\gamma+1} x^t =\cdots= \left(\frac{1-2\gamma}{2\gamma+1}\right)^{t+1} x^0.
  \]
Thus, for this example, the classical PR method converges faster than the classical DR method when $\gamma \in (0,1)$.

Moreover, let $\beta=2.5$ and $\gamma <\frac{\beta-2}{(\beta+1)^2L_F}=\frac1{49}$. Then, the proposed PR method \eqref{scheme0} reads  \begin{equation}
  \left\{
  \begin{split}
  &y^{t+1}=\argmin_{y} \left\{\frac{7}{2}\|y\|^2 + \frac{1}{2\gamma}\|y - x^t\|^2\right\}=\frac{1}{1+7\gamma} x^t, \\
&z^{t+1}=\argmin_{z} \left\{- \frac{5}{2}\|z\|^2 + \frac{1}{2\gamma}\|2y^{t+1} - x^t - z\|^2\right\}=\frac{1}{1-5\gamma}(2y^{t+1}-x^t),\\
&x^{t+1}=x^t+2(z^{t+1}- y^{t+1})=\left(1-\frac{4 \gamma}{(1-5 \gamma)(1+7\gamma)}\right)x^t.
\end{split}
\right.
\end{equation}
Note that, for $\gamma=0.01<\frac{\beta-2}{(\beta+1)^2L_F}=\frac1{49}$, we have
\[
0<1-\frac{4 \gamma}{(1-5 \gamma)(1+7\gamma)} \le 0.97 < \frac{1}{2\gamma+1}.
\]
Thus, for $\gamma=0.01$, our proposed PR method \eqref{scheme0} is faster than the classical DR method for this example.
\end{example}
 \begin{example}{\bf (classical PR method vs the proposed PR method)}
 Let $C=\{(0,0)\}$ and $D=\big(\{0\} \times \R\big) \cup \big(\R \times \{0\}\big)$. We consider the feasibility problem of finding a point in the intersection of $C$ and $D$.  We start with the initial point $x^0=(a,0)$ with $a\neq 0$. Then, the classical PR splitting method applies \eqref{scheme} to  $f(x)=\delta_C(x)$ and $g(x)=\delta_D(x)$ for all $x \in \R^2$, and
 reduces to
     \[
  x^{t+1}=(2{\rm prox}_{\gamma g} - I)\circ(2{\rm prox}_{\gamma f} - I)(x^t)=(2P_D - I)\circ(2P_C - I)(x^t)=-x^{t}.
  \]
 Thus, the classical PR splitting method diverges and cycles between two points $(a,0)$ and $(-a,0)$.
On the other hand, let $\beta=5$ and $\gamma \in \left(0,\frac{1}{12}\right)$ and consider the proposed PR method \eqref{scheme3} for feasibility problems. This algorithm reads

\begin{equation}\label{scheme6}
\left\{
\begin{split}
&y^{t+1} = \frac{\gamma P_C\left(\frac{x^t}{1 + \beta\gamma}\right) + x^t}{(1 + \beta)\gamma + 1}=\frac{x^t}{6\gamma + 1}, \\
&z^{t+1}\in P_D\left(\frac{2y^{t+1} - x^t}{1 - \beta\gamma}\right)=\left\{\frac{2y^{t+1} - x^t}{1 - 5\gamma}\right\},\\
&x^{t+1}=x^t+2(z^{t+1}- y^{t+1})=\left(1-\frac{2\gamma}{(1-5\gamma)(6\gamma+1)}\right)x^t,
\end{split}
\right.
\end{equation}
where the formula for the $z$-update follows from the fact that $x^t$, $y^t \in \R\times \{0\}\subset D$, and so is $2y^{t+1}-x^t$ by the construction.
Hence, the proposed PR method \eqref{scheme6} converges to $(0,0) \in C \cap D$ linearly in this case.
\end{example}

%\begin{example}{\bf (Proposed PR method vs classical DR method)}
%
%\end{example}
% \end{example}


\begin{thebibliography}{99}

  \bibitem{AtBoReSo10}
  H. Attouch, J. Bolte, P. Redont and A. Soubeyran.
  \newblock Proximal alternating minimization and projection methods for nonconvex problems. An approach based on the Kurdyka-{\L}ojasiewicz inequality.
  \newblock {\em Math. Oper. Res.} 35, pp. 438--457 (2010).


  \bibitem{AtBoSv13}
  H. Attouch, J. Bolte and B. F. Svaiter.
  \newblock Convergence of descent methods for semi-algebraic and tame problems: proximal algorithms, forward-backward splitting, and regularized Gauss-Seidel methods.
  \newblock {\em Math. Program.} 137, pp. 91--129 (2013).

  \bibitem{Jon1993}
 H.~H. Bauschke and J.~M. Borwein.
 \newblock On the convergence of von Neumann's alternating projection algorithm for two sets.
 \newblock {\em Set-Valued Anal.} 1, pp. 185--212 (1993).

  \bibitem{Bauschke_SIAM_review}
H.~H. Bauschke and J.~M. Borwein.
\newblock On projection algorithms for solving convex feasibility problems.
\newblock {\em SIAM Rev.} 38, pp. 367--426 (1996).


  \bibitem{BauCom11}
H.~H. Bauschke and P.~L. Combettes.
\newblock {\em Convex Analysis and Monotone Operator Theory in Hilbert Spaces}.
\newblock Springer (2011).

  \bibitem{BvdBSC13}
  M. Bogdan, E. van den Berg, W. Su and E. Cand\`es.
  \newblock Statistical estimation and testing via the sorted L1 norm.
  \newblock Preprint (2013). Available at \verb+http://arxiv.org/abs/1310.1969+.

  \bibitem{BolDanLew07}
  J.~Bolte, A.~Daniilidis and A.~Lewis.
  \newblock The {\L}ojasiewicz inequality for nonsmooth subanalytic functions with applications to subgradient dynamical systems.
  \newblock {\em SIAM J. Optim.}  17,  pp. 1205--1223 (2007).

  \bibitem{BolDanLewShi07}
  J.~Bolte, A.~Daniilidis, A.~Lewis and M. Shiota.
  \newblock Clarke subgradients of stratifiable functions.
  \newblock {\em SIAM J. Optim.}  18,  pp. 556--572 (2007).

    \bibitem{Li}
  J.~M.~Borwein, G.~Li and L.~J.~Yao.
  \newblock Analysis of the convergence rate for the cyclic projection algorithm applied to basic semialgebraic convex sets.
  \newblock {\em SIAM J. Optim.} 24, pp. 498--527 (2014).

    \bibitem{BDE}
A. M. Bruckstein, D. L. Donoho and M. Elad.
\newblock From sparse solutions of systems of equations
to sparse modeling of signals and images.
\newblock {\em SIAM Rev.} 51, pp. 34--81 (2009).

  \bibitem{CaTa05}
E.~Cand\`es and T.~Tao.
\newblock The Dantzig selector: statistical estimation when $p$ is much
larger than $n$.
\newblock {\em Ann. Statist.} 35, pp. 2313--2351 (2007).

%  \bibitem{Comb04}
%  P. L. Combettes.
%  \newblock Solving monotone inclusions via compositions of nonexpansive averaged operators.
%  \newblock {\em Optim.} 53, pp. 475--504 (2004).

  \bibitem{Comb09}
  P. L. Combettes.
  \newblock Iterative construction of the resolvent of a sum of maximal monotone operators.
  \newblock {\em J. Convex Anal.} 16, pp. 727--748 (2009).

\bibitem{Comb12}
P. L. Combettes and J.-C. Pesquet.
\newblock A Douglas-Rachford splitting approach to nonsmooth convex variational signal recovery.
\newblock {\em IEEE J. Sel. Topics Signal Process.} 1, no. 4, pp. 564--574, (2007).

\bibitem{Dobra09}
A. Dobra.
\newblock Variable selection and dependency networks for genomewide data.
\newblock {\em Biostatistics.} 10, pp. 621–-639 (2009).

\bibitem{PDE}
  J.~Douglas and H.~H.~Rachford.
\newblock  On the numerical solution of heat conduction problems in two or three space variables.
\newblock {\em T. Am. Math. Soc.} 82, pp. 421--439 (1956).
%
% \bibitem{Davis_Yin_2015}
% D. Davis and W. Yin,
% \newblock Convergence rates of relaxed Peaceman-Rachford and ADMM under regularity assumptions,
% \newblock{\em UCLA CAM Report 14-58}, 2014, Available at \verb+http://arxiv.org/pdf/1407.5210.pdf+.

  \bibitem{Eckstein_Bertsekas}
  J.~Eckstein and D.~P. Bertsekas.
  \newblock On the Douglas-Rachford splitting method and the proximal
point algorithm for maximal monotone operators.
\newblock {\em Math. Program.} 55, pp. 293--318 (1992).

\bibitem{FanJ}
J. Fan and R. Li.
\newblock Variable selection via nonconcave penalized
likelihood and its oracle properties.
\newblock {\em J. Amer. Statist. Assoc.} 96, pp. 1348--1360 (2001).

\bibitem{Boyd2015}
P. Giselsson and S. Boyd.
\newblock Diagonal scaling in Douglas-Rachford splitting and ADMM.
\newblock In {\em Proc. of the 53rd IEEE Conf. on Decision and Contr.}, pp. 5033--5039 (2014).


\bibitem{Golub99}
T. R. Golub,  D. K. Slonim, P. Tamayo, C. Huard, M. Gaasenbeek, J. P. Mesirov, H. Coller, M. L. Loh, J. R. Downing, M. A. Caligiuri, C. D. Bloomfield and E. S. Lander.
\newblock Molecular classification of cancer: class discovery and class prediction by gene expression monitoring.
\newblock {\em Science.} 286, pp. 531--537 (1999).

%\bibitem{GoldMaSch13}
%Donald Goldfarb, Shiqian Ma and Katya Scheinberg.
%\newblock Fast alternating linearization methods for minimizing the sum of two convex functions.
%\newblock {\em Math. Program. Ser. A} 141, pp. 349--382 (2013).

%  \bibitem{Han_Yuan15}
%  D.~R. Han and X.~M. Yuan.
%  \newblock Convergence analysis of the Peaceman-Rachford splitting method for nonsmooth convex optimization.
%  \newblock Preprint (2013). Available at \verb+http://www.optimization-online.org/DB\_FILE/2012/11/3689.pdf+.

\bibitem{Hong14}
M. Hong, Z.-Q. Luo and M. Razaviyayn.
\newblock Convergence analysis of alternating direction method of multipliers for a family of nonconvex problems.
\newblock {\em SIAM J. Optim.} 26, pp. 337--364 (2016).

  \bibitem{Lu12}
Z. Lu, T.K. Pong and Y. Zhang.
\newblock An alternating direction method for finding Dantzig selectors.
\newblock {\em Comput. Stat. Data An.} 56, pp. 4037--4046 (2012).




  \bibitem{KnF00}
K.~Knight and W.~Fu.
\newblock Asymptotics for the lasso-type estimators.
\newblock {\em Ann. Statist.} 28, pp. 1356--1378 (2000).

\bibitem{KBC12}
A. Kyrillidis, S. Becker, V. Cevher and C. Koch.
\newblock Sparse projections onto the simplex.
\newblock {\em JMLR W\&CP} 28, pp. 235--243 (2013).

  \bibitem{LiPong14}
  G. Li and T. K. Pong.
  \newblock Global convergence of splitting methods for nonconvex composite optimization.
  \newblock {\em SIAM J. Optim.} 25, pp. 2434--2460 (2015).

  \bibitem{LiPong14_2}
  G. Li and T. K. Pong.
  \newblock Douglas-Rachford splitting for nonconvex optimization with application to nonconvex feasibility problems.
  \newblock {\em Math. Program.} 159, pp. 371--401 (2016).

  \bibitem{LionsM79}
  P. L. Lions and B. Mercier.
  \newblock Splitting algorithms for the sum of two nonlinear operators.
  \newblock {\em SIAM J. Numer. Anal.} 16, pp. 964--979 (1979).

  \bibitem{Luke08}
  D. R. Luke,
  \newblock Finding best approximation pairs relative to a convex and a prox-regular set in a Hilbert space.
  \newblock {\em SIAM J. Optim.} 19, pp. 714--739 (2008).

  \bibitem{Luke1}
R. Hesse and D. R. Luke.
\newblock Nonconvex notions of regularity and convergence of fundamental algorithms for feasibility problems.
\newblock {\em SIAM J. Optim.} 23, pp. 2397--2419 (2013).

%  \bibitem{Luke2}
%R. Hesse, D. R. Luke and P. Neumann.
%\newblock Alternating projections and Douglas-Rachford for sparse affine feasibility.
%\newblock {\em IEEE T. Signal. Proces.} 62, pp. 4868--4881 (2014).

%  \bibitem{MRY07}
%N.~Meinshausen, G.~Rocha and B.~Yu.
%\newblock Discussion: a tale of
%three cousins: lasso, L2boosting and Dantzig.
%\newblock {\em Ann. Statist.} 35, pp. 2373--2384 (2007).

  \bibitem{PaStBe14}
  P. Patrinos, L. Stella and A. Bemporad.
  \newblock Douglas-Rachford splitting: complexity estimates and accelerated variants.
  \newblock In {\em Proc. of the 53rd IEEE Conf. on Decision and Contr.}, pp. 4234--4239 (2014).

  \bibitem{PR55}
  D. W. Peaceman and H. H. Rachford.
  \newblock The numerical solution of parabolic and elliptic differential equations.
  \newblock {\em J. Soc. Ind. Appl. Math.} 3, pp. 28--41 (1955).

%  \bibitem{Roc70}
%  R. T. Rockafellar.
%  \newblock {\em  Convex Analysis}.
%  \newblock Princeton University Press, Princeton (1970).

  \bibitem{Rock98}
  R. T. Rockafellar and R. J.-B. Wets.
  \newblock {\em Variational Analysis}.
  \newblock Springer (1998).

  \bibitem{Tib96}
R.~Tibshirani.
\newblock Regression shrinkage and selection via the lasso.
\newblock {\em J. R. Statist. Soc. B}  58, pp. 267--288 (1996).

%\bibitem{TSRZK05}
%R.~Tibshirani, M.~Saunders, S.~Rosset, J.~Zhu and K.~Knight.
%\newblock Sparsity and smoothness via the fused lasso.
%\newblock {\em J. R. Statist. Soc. B} 67, pp. 91--108 (2005).

\bibitem{WangBan13}
  H. Wang and A. Banerjee.
  \newblock Bregman alternating direction method of multipliers.
  \newblock {\em NIPS} 27, pp. 2816--2824 (2014).

\bibitem{Yeung05}
  K. Y. Yeung, R. E. Bumgarner and A. E. Raftery.
  \newblock Bayesian model averaging: development of an improved multi-class, gene selection and classification tool for microarray data. \newblock {\em Bioinformatics.} 21, pp. 2394--2402 (2005).

%   \bibitem{CHZhang}
%   C.-H. Zhang.
%\newblock Nearly unbiased variable selection under minimax
%concave penalty.
%\newblock {\em Ann. Statist.} 38, pp. 894--942 (2010).

% \bibitem{ZoH05}
% H.~Zou and T.~Hastie.
% \newblock Regularization and variable selection via the elastic net.
% \newblock {\em J. R. Statist. Soc. B} 67, pp. 301--320 (2005).

  \end{thebibliography}
\end{document}